\documentclass{amsart}

\usepackage{amssymb}
\usepackage{graphicx}

\newtheorem{theorem}{Theorem}[section]
\newtheorem{lemma}[theorem]{Lemma}
\newtheorem{proposition}[theorem]{Proposition}
\newtheorem{corollary}[theorem]{Corollary}

\newtheorem{_computation}[theorem]{Computation}
\newenvironment{computation}{\begin{_computation}\rm}{\end{_computation}}

\newtheorem{_definition}[theorem]{Definition}
\newenvironment{definition}{\begin{_definition}\rm}{\end{_definition}}

\newtheorem{_remark}[theorem]{\it Remark}
\newenvironment{remark}{\begin{_remark}\rm}{\end{_remark}}

\newtheorem{_example}[theorem]{Example}
\newenvironment{example}{\begin{_example}\rm}{\end{_example}}

\numberwithin{equation}{section}
\numberwithin{table}{section}
\numberwithin{figure}{section}

\newcommand{\C}{\mathord{\mathbb C}}

\newcommand{\F}{\mathord{\mathbb F}}

\renewcommand{\P}{\mathord{\mathbb  P}}
\newcommand{\Q}{\mathord{\mathbb  Q}}
\newcommand{\R}{\mathord{\mathbb R}}
\newcommand{\Z}{\mathord{\mathbb Z}}

\newcommand{\BBB}{\mathord{\mathcal B}}
\newcommand{\CCC}{\mathord{\mathcal C}}
\newcommand{\DDD}{\mathord{\mathcal D}}
\newcommand{\EEE}{\mathord{\mathcal E}}
\newcommand{\FFF}{\mathord{\mathcal F}}
\newcommand{\GGG}{\mathord{\mathcal G}}

\newcommand{\LLL}{\mathord{\mathcal L}}
\newcommand{\MMM}{\mathord{\mathcal M}}
\newcommand{\OOO}{\mathord{\mathcal O}}
\newcommand{\PPP}{\mathord{\mathcal P}}
\newcommand{\WWW}{\mathord{\mathcal W}}
\newcommand{\XXX}{\mathord{\mathcal X}}
\newcommand{\YYY}{\mathord{\mathcal Y}}
\newcommand{\ZZZ}{\mathord{\mathcal Z}}

\font\mathgot=eufm10

\newcommand{\SSSS}{\mathord{\hbox{\mathgot S}}}

\newcommand{\maprightsp}[1]{\; \smash{\mathop{\; \longrightarrow \; }\limits\sp{#1}}\; }

\newcommand{\mapdown}{\phantom{\Big\downarrow}\hskip -8pt \downarrow}
\newcommand{\mapdownright}[1]{\mapdown\rlap{$\vcenter{\hbox{$\scriptstyle#1$}}$}}
\newcommand{\mapdownleft}[1]{\rlap{$\vcenter{\hbox{$\scriptstyle#1$}}$}\mapdown}

\newcommand{\inj}{\hookrightarrow}
\newcommand{\isom}{\mathbin{\,\raise -.6pt\rlap{$\to$}\raise 3.5pt%
\hbox{\hskip .3pt$\mathord{\sim}$}\,}}

\newcommand{\set}[2]{\{\; {#1} \; \mid \; {#2} \;  \}}
\newcommand{\shortset}[2]{\{ {#1} \,|\, {#2}   \}}

\newcommand{\sprime}{\sp\prime}

\newcommand{\sptimes}{\sp{\times}}
\newcommand{\sperp}{\sp{\perp}}

\newcommand{\dual}{\sp{\vee}}
\newcommand{\inv}{\sp{-1}}

\newcommand{\NS}{\mathord{\rm NS}}

\newcommand{\Aut}{\operatorname{\rm Aut}\nolimits}

\newcommand{\Sing}{\operatorname{\rm Sing}\nolimits}
\newcommand{\sing}{\operatorname{\rm sing}\nolimits}
\newcommand{\Hom}{\operatorname{\rm Hom}\nolimits}

\newcommand{\rank}{\operatorname{\rm rank}\nolimits}
\newcommand{\disc}{\operatorname{\rm disc}\nolimits}
\newcommand{\ord}{\mathop{\rm ord}\nolimits}
\newcommand{\Pt}{\P^2}
\newcommand{\pione}{\pi_1}
\newcommand{\Pic}{\operatorname{\rm Pic}\nolimits}

\newcommand{\rmand}{\textrm{and}}
\newcommand{\rmor}{\textrm{or}}

\newcommand{\quand}{\quad\rmand\quad}

\newcommand{\tsum}{\textstyle\sum}

\newcounter{rmkakkocounter}
\newcommand{\setrmkakko}{\setcounter{rmkakkocounter}{1}}
\newcommand{\rmkakko}{{\rm (\thermkakkocounter)} \addtocounter{rmkakkocounter}{1}}

\newcommand{\tensor}{\otimes}

\newcommand{\XB}{X_B}
\newcommand{\YB}{Y_B}
\newcommand{\Gm}{\Gamma}
\newcommand{\tlGm}{\tilde\Gamma}
\newcommand{\tlGmplus}{\tilde\Gamma^+}
\newcommand{\tlGmminus}{\tilde\Gamma^-}
\newcommand{\tlGmsprimeplus}{\tilde\Gamma^{\prime+}}

\newcommand{\tlBi}{\tilde{B}_i}

\newcommand{\tlrhoB}{\tilde{\rho}_B}

\newcommand{\degs}{\mathop{\rm degs}}

\newcommand{\excEss}[1]{\EEE_{#1}}
\newcommand{\excEsB}{\excEss{B}}

\newcommand{\blat}{\Sigma} 
\newcommand{\bBlat}{\Theta}
\newcommand{\lat}{\Lambda}

\newcommand{\blatt}[1]{\blat_{#1}} 
\newcommand{\bBlatt}[1]{\bBlat_{#1}}
\newcommand{\latt}[1]{\lat_{#1}}

\newcommand{\blatB}{\blatt{B}} 
\newcommand{\bBlatB}{\bBlatt{B}}
\newcommand{\latB}{\latt{B}}

\newcommand{\configtype}{\gamma}
\newcommand{\latdata}{\ell}
\newcommand{\latdataZP}{\ell^P}
\newcommand{\lattype}{\lambda}
\newcommand{\lattypeZP}{\lambda^P}

\newcommand{\QC}{{\rm{QC}}}

\newcommand{\gen}[1]{\langle#1\rangle}

\newcommand{\releqs}{\sim_{\rm{eqs}}}
\newcommand{\rellat}{\sim_{\rm{lat}}}
\newcommand{\relconfig}{\sim_{\rm{cfg}}}
\newcommand{\relemb}{\sim_{\rm{emb}}}

\newcommand{\nreleqs}{{\not\sim}_{\rm{eqs}}}
\newcommand{\nrellat}{{\not\sim}_{\rm{lat}}}

\newcommand{\nrelemb}{{\not\sim}_{\rm{emb}}}

\newcommand{\intnum}[2]{(#1,#2)}

\newcommand{\torus}{\mathord{\rm{trs}}}
\newcommand{\Btorus}{B_{\torus}}

\newcommand{\AAAA}{\mathord{\hbox{\mathgot A}}}
\newcommand{\BBBB}{\mathord{\hbox{\mathgot B}}}
\newcommand{\CCCC}{\mathord{\hbox{\mathgot C}}}
\newcommand{\DDDD}{\mathord{\hbox{\mathgot D}}}

\newcommand{\aaaa}{\mathord{\hbox{\mathgot a}}}
\newcommand{\bbbb}{\mathord{\hbox{\mathgot b}}}
\newcommand{\cccc}{\mathord{\hbox{\mathgot c}}}
\newcommand{\dddd}{\mathord{\hbox{\mathgot d}}}
\newcommand{\eeee}{\mathord{\hbox{\mathgot e}}}
\newcommand{\ffff}{\mathord{\hbox{\mathgot f}}}

\font\smaleufm=eufm8
\newcommand{\smallAAAA}{\mathord{\hbox{\smaleufm A}}}
\newcommand{\smallBBBB}{\mathord{\hbox{\smaleufm B}}}
\newcommand{\smallCCCC}{\mathord{\hbox{\smaleufm C}}}
\newcommand{\smallDDDD}{\mathord{\hbox{\smaleufm D}}}
\newcommand{\smallaaaa}{\mathord{\hbox{\smaleufm a}}}
\newcommand{\smallbbbb}{\mathord{\hbox{\smaleufm b}}}
\newcommand{\smallcccc}{\mathord{\hbox{\smaleufm c}}}
\newcommand{\smalldddd}{\mathord{\hbox{\smaleufm d}}}

\newcommand{\Klat}{\mathord{\mathbb L}}

\newcommand{\polB}{h_B}

\newcommand{\rlat}{L}

\newcommand{\zeroGamma}{{}^0\Gamma}
\newcommand{\Kahler}{\mathord{\mathcal K}}
\newcommand{\spzero}{{}^0 \hskip -1pt }

\newcommand{\pol}{\LLL}

\newcommand{\per}{\omega}
\newcommand{\persp}{\Omega}
\newcommand{\tR}{\tensor\R}

\newcommand{\Ball}{\mathbb{B}}

\newcommand{\tlE}{\tilde{E}}
\newcommand{\tlB}{\tilde{B}}

\newcommand{\LD}{\mathord{\tt LD}}
\newcommand{\ttlatdataZP}{{\mathord{l}}^{P}}

\newcommand{\perspcond}{\diamond}

\newcommand{\mystruth}[1]{\phantom{\hbox{\vrule height #1}}}
\newcommand{\mystrutd}[1]{\phantom{\hbox{\vrule depth #1}}}
\newcommand{\mystruthd}[2]{\phantom{\hbox{\vrule  height #1 depth #2}}}

\begin{document}

\title[Lattice Zariski $k$-ples]{Lattice Zariski $k$-ples 
of plane sextic curves and $Z$-splitting curves for double plane sextics}

\author{Ichiro Shimada}
\address{
Department of Mathematics, 
Graduate School of Science, 
Hiroshima University,
1-3-1 Kagamiyama, 
Higashi-Hiroshima, 
739-8526 JAPAN
}
\email{shimada@math.sci.hiroshima-u.ac.jp
}

\thanks{Partially supported by
 JSPS Grants-in-Aid for Scientific Research (20340002) and 
 JSPS Core-to-Core Program (18005).}



\begin{abstract}
A simple sextic is 
a reduced  complex projective plane curve   of degree $6$
with only simple singularities.
We introduce a notion of $Z$-splitting curves
for the double covering of the projective plane branching along a simple sextic,
and  investigate lattice  Zariski $k$-ples of simple sextics by means of this notion.
Lattice types of $Z$-splitting curves and their  specializations are defined. 
All  lattice types of  $Z$-splitting curves of degree less than or equal to $3$ are classified 
up to specializations.
 \end{abstract}

\maketitle

\section{Introduction}\label{sec:Introduction}
In virtue of the theory of period mapping,
the lattice theory has become a strong computational tool 
in the study of complex $K3$ surfaces.
In this paper,
we apply  this tool to the classification of complex projective plane  curves of degree $6$
with only simple singularities.
In particular, we explain the phenomena of \emph{Zariski pairs}
from lattice-theoretic point of view.
\par
\medskip
A \emph{simple sextic} is 
a reduced  (possibly reducible) complex projective plane curve   of degree $6$
with only simple singularities.
For a simple sextic $B\subset \P^2$,
we denote by $\mu_B$ the total Milnor number of $B$,
by $\Sing B$ the singular locus of $B$, 
by $R_B$ the $ADE$-type of the singular points  of $B$,
and by $\degs B=[d_1, \dots, d_m]$ the list of degrees $d_i=\deg B_i$ of the irreducible components
$B_1, \dots, B_m$ of $B$.
\par
\medskip
We have the following equivalence relations among simple sextics.
\begin{definition}\label{def:firstthreerels}
Let $B$ and $B\sprime$ be simple sextics.

(1) We write $B\releqs B\sprime$ if $B$ and $B\sprime$ are 
contained in the same connected component
of an equisingular family of simple sextics.

(2) We say that 
$B$ and $B\sprime$ 
are of  the \emph{same configuration type} and write $B\relconfig  B\sprime$
 if there exist  tubular neighborhoods
 $T\subset \Pt$ of $B$  and  $T\sprime\subset \Pt$ of $B\sprime$
 and a homeomorphism  $\varphi: (T, B)\isom (T\sprime, B\sprime)$
 such that $\deg \varphi(B_{i})=\deg B_{i}$ holds
 for each irreducible component $B_{i}$ of $B$,
that $\varphi$ induces a bijection $\Sing B\isom \Sing B\sprime$,  
and that $\varphi$ is an analytic isomorphism of plane curve singularities
 locally  around  each $P\in \Sing B$.
Note that $R_B$ and $\degs B$ are invariants of 
the configuration type.
(See \cite[Remark 3]{MR2409555} 
for a combinatorial definition of $\relconfig$.)

(3) We say that 
$B$ and $B\sprime$ 
are of  the \emph{same embedding type} and write $B\relemb  B\sprime$
 if there exists  a homeomorphism  $\psi: (\Pt, B)\isom (\Pt, B\sprime)$
 such that 
$\psi$ induces a bijection $\Sing B\isom \Sing B\sprime$ and that,
locally around  each $P\in \Sing B$, $\psi$ is an analytic isomorphism of  plane curve singularities.
\end{definition}
It is obvious that 
$$
B\releqs B\sprime\; \Longrightarrow\; B\relemb B\sprime \; \Longrightarrow\; B\relconfig B\sprime,
$$
while the converses do not necessarily hold.
\begin{example}\label{example:Z1}
Zariski~\cite{MR1506719} showed that there exist irreducible simple sextics $B_1$ and $B_2$ with $R_{B_1}=R_{B_2}=6A_2$
such that  $\pione (\Pt\setminus B_1)\cong \Z/2\Z\times \Z/3\Z$ while  $\pione (\Pt\setminus B_2)\cong  \Z/2\Z*  \Z/3\Z$,
where $*$ denotes the free product of groups. 
(See also Oka~\cite{MR1167373} and Shimada~\cite{MR1421396}).
Therefore we have $B_1\relconfig B_2$, 
but $B_1\nrelemb B_2$ and hence $B_1\nreleqs B_2$.
\end{example}
Artal-Bartolo~\cite{MR1257321} revived the study of pairs of plane curves that are 
of the same configuration type but 
are not connected by equisingular deformation.
Since then,
many works have been done about the discrepancies 
between  equisingular deformations  and  configuration types,
not necessarily for simple sextics 
but also for curves of higher degrees and with other types of singularities.
(See the survey paper~\cite{MR2409555}.)
The main theme of these works is  to find  pairs of plane curves 
(called \emph{Zariski pairs} or \emph{Zariski couples})
that have the same configuration type but have different embedding topologies.
%
\par
\medskip
As for simple sextics,
there have been two important works about 
  $\releqs$ and  $\relconfig$;
one is Yang~\cite{MR1387816},
in which the configuration types of simple sextics are completely classified,
and the other is Degtyarev~\cite{MR2357681},
in which an algorithm to calculate the connected components 
of the equisingular family of simple sextics in a given configuration type
is presented.
The main tool of these two works  is the
theory of period mapping of complex $K3$ surfaces
applied to double plane sextics.
\par
\medskip
In this paper, 
we introduce another equivalence relation $\rellat$
by means of the structure of the N\'eron-Severi lattices of the $K3$ surfaces
obtained as the double covers of $\Pt$ branching along the simple sextics.
This relation  is coarser than $\releqs$ but finer than $\relconfig$,
and hence can play the same role as $\relemb$.
The definition of $\rellat$ is, however, purely algebraic
and therefore computationally easier to deal with than $\relemb$.
In fact, Yang's method~\cite{MR1387816} provides us with an algorithm to classify 
all the equivalence classes of the relation $\rellat$,
which are called the \emph{lattice types} of simple sextics.
Moreover we can sometimes conclude $B\nrelemb B\sprime$
by looking at an invariant of the lattice types (Theorem~\ref{thm:lat_and_emb}).

We then define the notion of \emph{$Z$-splitting curves},
and investigate lattice types of simple sextics by means of this notion.
A notion of \emph{lattice Zariski couples} 
(or more generally, \emph{lattice Zariski $k$-ples}) is introduced for $\rellat$
in the same way as  the notion of classical Zariski couples was introduced for $\relemb$ in~\cite{MR1257321}.
The notion of  $Z$-splitting curves provides us with a unifying tool 
to describe  all lattice Zariski $k$-ples.
In fact, the members of any lattice Zariski $k$-ple are distinguished by 
numbers of $Z$-splitting curves of degree $\le 2$ (Theorem~\ref{thm:LZkplets}).

Finally, we define \emph{lattice types of $Z$-splitting curves}, and 
classify all lineages via specialization of lattice types of $Z$-splitting curves of degree $\le 3$.
It turns out that these lineages are completely indexed by the 
\emph{class-order} of the $Z$-splitting curves (Theorems~\ref{thm:Zlines},~\ref{thm:Zconics} and~\ref{thm:sp}).
These lineages seem to yield 
many examples of simple sextics with interesting geometry.
For example,
the $Z$-splitting conics with class-order $3$
are the splitting conics of \emph{torus sextics},
which have been studied intensively by Oka and others~(see \cite{MR1948673}, for example).

Another importance of $Z$-splitting curves comes from the fact that,
for a  simple sextic $B$ that is generic in  an irreducible component of an equisingular family,
the N\'eron-Severi lattice of the corresponding $K3$ surface
is generated by the reduced parts of the lifts of the irreducible components of $B$ and 
the lifts of $Z$-splitting curves of degree $\le 3$
(Theorem~\ref{thm:FB}).
\par
\medskip
The plan of this paper is as follows.
In \S\ref{sec:definitions},
we define various notions that are investigated in this paper.
The relation $\rellat$ is defined in Definition~\ref{def:rellat}, 
and the notion of $Z$-splitting curves is defined in Definition~\ref{def:Zsplitting}.
The main results are stated in~\S\ref{sec:mainresults}.
Most of these results are proved computationally
with assistance of a computer.
We present lattice-theoretic algorithms to prove them
in the following sections.
In \S\ref{sec:Yang}, we  explain the  method of Yang 
to make the complete list of lattice types of simple sextics.
In \S\ref{sec:algorithm},
we give an algorithm to determine
the configuration type and the classes of lifts of smooth $Z$-splitting curves of degree $\le 3$
for a given lattice type of simple sextics.
In \S\ref{sec:specialization},
we present an algorithm
about specialization of lattice types of $Z$-splitting curves.
Results in \S\ref{sec:specialization} are the main theoretical ingredients
for our classification of the lineages of $Z$-splitting curves.
In~\S\ref{sec:demonstration},
we demonstrate the algorithms for a concrete example.
We conclude this paper
by presenting miscellaneous facts, examples and remarks in~\S\ref{sec:remarks}.
\par
\medskip
When we were finishing the first version of this paper,
a preprint by Yang and  Xie~\cite{discYang} appeared on the e-print archive.
In their paper,  Yang and  Xie also investigate the classical Zariski pairs of simple sextics by lattice theory
and the result in~\cite{nonhomeo, MR2405237}. See also Theorem~\ref{thm:lat_and_emb} of this paper.
\par
\medskip
{\bf Acknowledgement.} Part of this work was done during the author's stay at National University of Singapore
in September 2008.
Thanks are due to Professor De-Qi Zhang  for his warm hospitality.
The author is  also deeply grateful to  the referee for many valuable comments on the first version of this paper.
\section{Definitions}\label{sec:definitions}
A \emph{lattice} is a free $\Z$-module $L$ of finite rank
with a non-degenerate symmetric bilinear form 
$\intnum{\phantom{i}}{\phantom{i}}: L\times L\to \Z$.
We say that a lattice $L$ is \emph{even} if $x^2\in 2\Z$ holds for any $x\in L$. 
We say that $L$ is \emph{negative-definite} if $x^2<0$ holds for any non-zero $x\in L$.
\par
We fix several conventions  about lattices.
Let $L$ be a lattice, and let $S$ be a subset of $L$.
We denote by $\gen{S}$ the sublattice of $L$ generated by $S$ and by $\gen{S}\sp+$
the \emph{monoid} of vectors $\sum a_v v$ $(v\in S)$ with $a_v\in \Z_{\ge 0}$.
When $S=\{v\}$, we write $\gen{v}$ for $\gen{\{v\}}$.
We denote by $S\sp{\perp}$ or $(S\subset L)\sp\perp$
the orthogonal complement of $\gen{S}$ in $L$.
\par
Let $L\sprime$ be another lattice.
An \emph{embedding} of $L$ into $L\sprime$
is a homomorphism of $\Z$-modules $\phi:L\to L\sprime$
that satisfies $(x, y)=(\phi(x), \phi(y))$ 
for any $x, y\in L$.
Note that such a homomorphism is necessarily injective.
An embedding $\phi$ is said to be \emph{primitive}
if the cokernel of $\phi$ is torsion free.
For  an embedding $\phi$, 
we use the same letter $\phi$ to denote the induced linear homomorphism
$L\tensor \C\to L\sprime\tensor \C$.
%
%
%
%
%
\begin{definition}
Let $L$ be an even negative-definite lattice.
A vector $d\in L$ is called a \emph{root} if $d^2=-2$.
Let $D_L$ be the set of roots in $L$.
A subset $F$ of $D_L$ is called a \emph{fundamental system of roots in $L$}
if  $F$ is a basis of $\gen{D_L}$
and every $d\in D_L$ can be written as a linear combination of elements of $F$
with coefficients all non-positive or all non-negative.
An even negative-definite lattice $L$ is called a \emph{root lattice}
if $\gen{D_L}=L$ holds.
\end{definition}
A fundamental system of roots   exists for any even negative-definite lattice.
The isomorphism classes of 
fundamental systems of roots (and hence root lattices) are classified by means of \emph{Dynkin diagrams}.
See Ebeling~\cite[\S1.4]{MR1938666} or Bourbaki~\cite{MR0240238}, for example, for the proof of these facts.
\par
\medskip
We denote by $L\dual$ the \emph{dual lattice}
$\shortset{v\in  L\tensor \Q}{\intnum{x}{v}\in \Z\;\textrm{for any}\; x\in L}$
of $L$, which is a submodule of $L\tensor \Q$ with finite rank containing $L$.
\begin{definition}
Let $L$ be a lattice.
A submodule $L\sprime$ of $L\dual$  is called 
an \emph{overlattice} of $L$ if 
$L\sprime$ contains $L$ and  the $\Q$-valued symmetric bilinear form on $L\dual$
extending 
the  symmetric bilinear form on $L$
takes values in $\Z$ on  $L\sprime$.
\end{definition}
\begin{definition}\label{def:lattice-type}
\emph{Lattice data}  is a triple
$[\EEE, h, \lat]$,
where $\EEE$ is a fundamental system of roots  in 
the negative-definite root lattice  $\gen{\EEE}$
generated by $\EEE$,
$h$ is a vector with $h^2=2$ that generates a positive-definite lattice  $\gen{h}$ of rank $1$,
and $\lat$ is an even overlattice of the orthogonal direct-sum $\gen{h}\oplus\gen{\EEE}$.

 \emph{Extended lattice data}  is a quartet 
$[\EEE, h, \lat, S]$,
where $[\EEE, h, \lat]$ is  lattice data and $S$ is a subset of $\lat$
with cardinality $2$.
\end{definition}
\begin{remark}
In the geometric application,
$S$ is the place holder for the classes of the lifts of a $Z$-splitting curve.
(See Definition~\ref{def:latdataZsplpair}.) 
\end{remark}
\begin{definition}\label{def:isom-lattice-type}
An \emph{isomorphism} from  lattice data $[\EEE, h, \lat]$ 
to  lattice data $[\EEE\sprime, h\sprime, \lat\sprime]$
is an isomorphism of lattices $\phi: \lat\isom \lat\sprime$  that satisfies $\phi(\EEE)=\EEE\sprime$ and $\phi(h)=h\sprime$.
If $\phi: \lat\isom \lat\sprime$ is an isomorphism of lattice data,
then $\phi$ induces an isomorphism of fundamental systems of roots between $\EEE$ and $\EEE\sprime$.
\end{definition}
\begin{definition}\label{def:isom-lattice-typeZP}
An \emph{isomorphism} from  extended lattice data $[\EEE, h, \lat, S]$ 
to extended lattice data $ [\EEE\sprime, h\sprime, \lat\sprime, S\sprime]$
is an isomorphism $\phi: \lat\isom \lat\sprime$ of lattice data from $[\EEE, h, \lat]$ to $[\EEE\sprime, h\sprime, \lat\sprime]$ 
that induces a bijection from  $S$ to $S\sprime$.
\end{definition}
Let $B\subset \P^2$ be a simple sextic.
Consider  the double covering $\pi_B: \YB\to \P^2$ branching exactly along  $B$.
Then $\YB$ has only rational double points of type $R_B$ as its singularities,
and the minimal resolution $\rho_B: \XB\to\YB$ of $\YB$
yields a $K3$ surface $\XB$.
Let $\tlrhoB: X_B\to\Pt$ denote the composite of $\rho_B$ and $\pi_B$.
\par
\medskip
We denote by 
$\NS(\XB)\subset H^2(\XB,\Z)$
 the N\'eron-Severi lattice of $X_B$.
Let $\excEsB$ be the set of $(-2)$-curves on $\XB$ that are contracted  by $\tlrhoB: X_B\to \Pt$.
We regard $\EEE_B$ as a subset of $\NS(\XB)$ by $E\mapsto [E]$,
where $[E]\in \NS(\XB)$ denotes the class of the curve $E\in \EEE_B$.
We consider the sublattice 
$$
\blatB:=\gen{\polB} \oplus \gen{\excEsB}\;\subset\;\NS(\XB)
$$
of $\NS(\XB)$ generated by the polarization class 
$$
\polB:=[\tlrhoB\sp * (\OOO_{\Pt}(1))]
$$
and $\excEsB\subset \NS(X_B)$.
Remark that $\EEE_B$ is a fundamental system of roots 
in the root lattice $\gen{\EEE_B}$ of type $R_B$.
We  then denote by
$$
\latB:=(\blatB\otimes\Q)\cap H^2(\XB,\Z)
$$
the primitive closure of $\blatB$ in $H^2(\XB,\Z)$, 
which is an even overlattice of $\blatB$.
Since $\NS(X_B)$ is primitive in $H^2(X_B, \Z)$,
$\latB$ is the primitive closure of $\blatB$ in $\NS(X_B)$.
Finally  we define the finite abelian group
$G_B$ by 
$$
G_B:=\latB/\blatB.
$$
\begin{definition}
We denote by $\latdata (B)$ the lattice data $[\EEE_B, \polB, \latB]$,
and call it the \emph{lattice data of $B$}.
\end{definition}
\begin{definition}\label{def:rellat}
Let $B$ and $B\sprime$ be simple sextics.
We  write $B\rellat  B\sprime$
 if there exists an isomorphism between the lattice data $\latdata (B)$
 and $\latdata (B\sprime)$.
An equivalence class of the relation $\rellat$ is called a \emph{lattice type} of simple sextics.
The lattice type containing a simple sextic $B$ is denoted by $\lattype (B)$.
\end{definition}
By definition,
an isomorphism of lattice data from $\latdata (B)$ to $\latdata (B\sprime)$
is an isomorphism of lattices $\latB\isom \lat_{B\sprime}$ that preserves the
polarization class and  the set of classes of the exceptional $(-2)$-curves.
\par
\medskip
It is obvious that the isomorphism class of the finite abelian group $G_B$ is 
an invariant of the lattice type $\lattype(B)$.
\par
\medskip
Let $B_1, \dots, B_m$ be the irreducible components of $B$.
We denote by $\tilde{B}_i\subset \XB$ the reduced part of the strict transform of
$B_i$, and put  
$$
\bBlatB:=\blatB+\gen{[\tilde{B}_1], \dots, [\tilde{B}_m]} \;\;\subset\;\; \NS(\XB).
$$
Then we have  
$$
\blatB\subset \bBlatB\subset \latB \subset \NS(X_B).
$$
We see that the implications 
$$
B\releqs B\sprime\;\; \Longrightarrow\;\; B\rellat B\sprime \;\; \Longrightarrow\;\; B\relconfig B\sprime
$$
hold, 
where the second implication was proved by Yang~\cite{MR1387816}.
(See also Corollary~\ref{cor:Yang}).
Hence the isomorphism class of 
the finite abelian group
$$
F_B:=\latB/\bBlatB
$$
is also an invariant of the lattice type $\lattype(B)$. 
\par
\medskip 
In fact, Yang~\cite{MR1387816} gave an algorithm to classify 
all lattice types and configuration types of simple sextics
using the idea of Urabe~\cite{MR1101859, MR1000608}. 
The numbers of these types are given  in Table~\ref{table:numbs}.
\begin{table}
\newcommand{\numbox}[1]{\hbox to 5mm {\hss$#1$\hss}}
$$
\begin{array}{c|ccccccccccccc}
\renewcommand{\arraystretch}{1.2}
\mu_B&
\numbox{0}&
\numbox{1}& 
\numbox{2}& 
\numbox{3}& 
\numbox{4}& 
\numbox{5}& 
\numbox{6}& 
\numbox{7}& 
\numbox{8}& 
\numbox{9}& 
\numbox{10}& 
\numbox{11}& 
\\ 
\hline
\relconfig&
\numbox{1}&
\numbox{1}& 
\numbox{2}& 
\numbox{3}& 
\numbox{6}& 
\numbox{10}& 
\numbox{18}& 
\numbox{30}& 
\numbox{53}& 
\numbox{89}& 
\numbox{148}& 
\numbox{246}& 
\\ 
\rellat &
\numbox{1}&
\numbox{1}& 
\numbox{2}& 
\numbox{3}& 
\numbox{6}& 
\numbox{10}& 
\numbox{18}& 
\numbox{30}& 
\numbox{53}& 
\numbox{89}& 
\numbox{148}& 
\numbox{246}& 
\\ 
\end{array}
$$
\vskip 5pt 
\renewcommand{\numbox}[1]{\hbox to 8mm {\hss$#1$\hss}}
$$
\begin{array}{c|cccccccc|c}
\renewcommand{\arraystretch}{1.2}
\mu_B&
\numbox{12}& 
\numbox{13}& 
\numbox{14}& 
\numbox{15}& 
\numbox{16}& 
\numbox{17}& 
\numbox{18}& 
\numbox{19}& 
\textrm{total} \\ 
\hline
\relconfig&
\numbox{415}& 
\numbox{684}& 
\numbox{1090}& 
\numbox{1623}& 
\numbox{2139}& 
\numbox{2283}& 
\numbox{1695}& 
\numbox{623}& 
\numbox{11159} \\ 
\rellat &
\numbox{416}& 
\numbox{686}& 
\numbox{1096}& 
\numbox{1639}& 
\numbox{2171}& 
\numbox{2330}& 
\numbox{1734}& 
\numbox{629}& 
\numbox{11308} \\ 
\end{array}
$$
\vskip 15pt 
\caption{Numbers of  configuration types and  lattice types}\label{table:numbs}
\end{table}
(Yang did not present the complete table in his paper,
and hence we re-produced the classification table by ourselves
along with  the complete list of configurations of rational double points on
normal $K3$ surfaces in~\cite{MR2369942}.)
Table~\ref{table:numbs} shows that, for $\mu_B>11$,
there exist many lattice Zariski $k$-ples
 $(k>1)$,
which is defined as follows.
\begin{definition}\label{def:LZP}
A  configuration type $\gamma$ of simple sextics is called a \emph{lattice Zariski $k$-ple}
if $\gamma$ contains exactly $k$ lattice types.
\end{definition}
\begin{example} \label{example:Z2}
The configuration type  of irreducible simple sextics $B$ 
with $R_B=6A_2$ is a lattice Zariski couple with $\mu_B=12$.
Indeed,  for $B_1$ and $B_2$ in Example~\ref{example:Z1}, we have  
$G_{B_1}=0$ while $G_{B_2}\cong \Z/3\Z$.
\end{example}
\begin{remark}
See~\S\ref{sec:demonstration} for an example of 
lattice Zariski triples.
Looking at the classification table, we see that
there exist no lattice Zariski $k$-ples with $k>3$.
\end{remark}
Next we define the notion of \emph{$Z$-splitting curves},
where $Z$ stands for \emph{Zariski}.
Let $B$ be a simple sextic.
We denote by
$$
{\iota_B} : X_B\isom  X_B
$$
 the involution of $X_B$ over $\Pt$,
 and use the same letter $\iota_B$
 to denote the induced involution on 
 the lattice  $H^2(X_B, \Z)$.
 Note that $\iota_B$ preserves the sublattices
  $\blat_B$, $\latB$, $\bBlatB$ and  $\NS(X_B)$.
\begin{definition}
A reduced irreducible projective plane curve $\Gm\subset \Pt$ is said to be \emph{splitting for $B$}
if the strict transform of $\Gm$ by
$\tlrhoB: \XB\to\Pt$ 
splits into  two  (possibly equal) irreducible components
$\tlGmplus$ and $\tlGmminus= \iota_B (\tlGmplus)$.
We call $\tlGmplus$ and $\tlGmminus$ 
the \emph{lifts} of the splitting curve $\Gm$.
\end{definition}
We have $\tlGmplus=\tlGmminus$ if and only if $\Gm$ is an irreducible component of $B$.
\begin{definition}\label{def:Zsplitting}
A splitting curve $\Gm$ is said to be \emph{pre-$Z$-splitting}
if the class $[\tlGmplus]$ of a lift  $\tlGmplus\subset \XB$
 of $\Gm$ is contained in $\latB$.
 (Note  that
$[\tlGmplus]\in \latB$ if and only if $[\tlGmminus]\in \latB$,
because we have $[\tlGmplus]+[\tlGmminus]\in \blatB$.)
\end{definition}
\begin{definition}
A pre-$Z$-splitting curve $\Gm$ is said to be \emph{$Z$-splitting}
if the classes $[\tlGmplus]$ and $[\tlGmminus]=\iota_B([\tlGmplus])$ are distinct.
\end{definition}
\begin{remark}
Since $\iota_B$ acts on the orthogonal complement of $\latB$ in $H^2(X_B, \Z)$
as the multiplication by $-1$,
it follows  that, if a splitting curve $\Gm$ is not pre-$Z$-splitting,
then we have $[\tlGmplus]\ne [\tlGmminus]$.
See Table~\ref{table:preZ}.
\end{remark}
\begin{table}
\begin{tabular}{c|cc}
&$[\tlGmplus]\in \latB$ & $[\tlGmplus]\notin \latB$ \mystruthd{12pt}{7pt}\\
\hline 
$[\tlGmplus]= [\tlGmminus]$&  $I$ & $\emptyset$\mystruthd{11pt}{7pt} \\
\hline 
$[\tlGmplus]\ne  [\tlGmminus]$ &  $II$ & $III$ \mystruth{12pt}
\end{tabular}
\hskip .5cm 
\begin{tabular}{lcl}
splitting&:& $I+II+III$\\
pre-$Z$-splitting&:& $I+II$\\
$Z$-splitting&:& $II$
\end{tabular}
\vskip 10pt 
\caption{Three notions of splittingness}\label{table:preZ}
\end{table}
We  have an easy numerical criterion of pre-$Z$-splittingness (see Proposition~\ref{prop:criterion}).
We also have  the following:
\begin{proposition}\label{prop:stable}
Let $\Gm$ be a pre-$Z$-splitting curve for a simple sextic $B$.
Let $B\sprime$ be a general member of the connected component $\FFF$ of 
the equisingular family containing $B$, 
and let $\phi:H^2(X_B, \Z)\isom H^2(X_{B\sprime}, \Z)$ be an  isomorphism of lattices
induced by an equisingular deformation from $B$ to $B\sprime$.
Then there exists a pre-$Z$-splitting curve $\Gm\sprime$ for $B\sprime$
such that the class of a lift of $\Gm\sprime$ is equal to $\phi([\tlGmplus])$.
If $\Gm$ is $Z$-splitting, then so is $\Gm\sprime$.
\end{proposition}
\begin{proof}
Since $\phi$ is induced by an equisingular deformation,
we see that $\phi$ induces an isomorphism $ \latB\isom \lat_{B\sprime}$.
The second assertion follows from the first assertion because $\phi$ 
commutes with the involutions $\iota_B$ and $\iota_{B\sprime}$.
Since $\Gm$ is irreducible by definition,
the lift $\tlGmplus$ is also irreducible and hence
we have $H^1(X_B, \OOO(\tlGmplus))=0$ by~\cite[Lemma 3.5]{MR0364263}.
Since $B\sprime$ is general in $\FFF$, 
we see that $\tlGmplus$ is deformed to an effective divisor $\tlGmsprimeplus$
on $X_{B\sprime}$ (see Lemmas~\ref{lem:LLL}~and~\ref{lem:V}),
and that $\tlGmsprimeplus$ is irreducible
and mapped birationally to a curve $\Gm\sprime$ on $\Pt$.
Hence $\phi([\tlGmplus])$ is the class of a lift $\tlGmsprimeplus$ of 
a splitting curve $\Gm\sprime$ for $B\sprime$.
Since $\phi([\tlGmsprimeplus])\in \lat_{B\sprime}$,
$\Gm\sprime$ is pre-$Z$-splitting.
\end{proof}
\begin{example}\label{example:tripletangent}
Let $f(x_0, x_1, x_2)$ and $g(x_0, x_1, x_2)$ be  general homogeneous polynomials of degree
$5$ and $3$, respectively.
Then  $B=\{x_0 f+g^2=0\}$ is smooth,
and the triple tangent line $\Gm=\{x_0=0\}$ is splitting for $B$
but not pre-$Z$-splitting,
because a general sextic has no triple tangents.
\end{example}
\begin{example}
Every irreducible component of $B$ is pre-$Z$-splitting, but not $Z$-splitting.
\end{example}
\begin{example} 
Suppose that $B$ is a union of cubic curves $E_0$ and $E_\infty$.
Then  the general member $E_t$ of the pencil in $|\OOO_{\Pt}(3)|$
spanned by $E_0$ and $E_\infty$ is   pre-$Z$-splitting.
The  lifts $\tlE_{t}^+$ and  $\tlE_{t}^-$  of $E_t$ are,
however,  contained in the same elliptic pencil on $X_B$,
and hence $E_t$ is not $Z$-splitting.
\end{example}
If a pre-$Z$-splitting curve $\Gm$ is of degree $\le 2$ and not contained in $B$,
then  its  lifts $\tlGmplus$ and  $\tlGmminus$ are distinct $(-2)$-curves on $X_B$,
and hence $\Gm$ is $Z$-splitting.
\begin{example}\label{example:Z3}
Let $f(x_0, x_1, x_2)$ and $g(x_0, x_1, x_2)$ be  general homogeneous polynomials of degree
$2$ and $3$, respectively.
Then the \emph{torus sextic} $\Btorus:=\{f^3+g^2=0\}$ 
is a simple sextic with $R_{\Btorus}=6 A_2$, and 
the conic  $\Gm=\{f=0\}$  is  $Z$-splitting,
as can be seen by
the numerical criterion Proposition~\ref{prop:criterion}. 
(See Example~\ref{example:numtorus}.)
In fact, this torus sextic  $\Btorus$ is the simple sextic $B_2$  in Examples~\ref{example:Z1}~and~\ref{example:Z2},
and the class $[\tlGmplus]$ generates the cyclic group $G_{B_2}=G_{\Btorus}$  of order $3$.
\end{example}
\begin{definition}\label{def:lattice-generic}
A simple sextic $B$ is said to be \emph{lattice-generic}
if  $\latB=\NS(\XB)$ holds,
or equivalently,
the Picard number  of $X_B$ is equal to $\mu_B+1$.
\end{definition}
\begin{remark}\label{rem:lattice_generic}
It is easy to see that  lattice-generic simple sextics 
are dense in any equisingular family.
(See Corollary~\ref{cor:perturb}.)
In particular,
every lattice type contains a lattice-generic member.
\end{remark}
\begin{corollary}
A splitting curve $\Gm$ for a simple sextic $B$ is  pre-$Z$-splitting
if and only if $\Gm$ is stable under  general equisingular deformation of $B$.
\end{corollary}
%
\begin{proof}
The ``\,only if\," part follows from Proposition~\ref{prop:stable}.
The ``\,if\," part follows from Remark~\ref{rem:lattice_generic}.
\end{proof}
The assumption that  $B\sprime$ be a general member of $\FFF$ in  Proposition~\ref{prop:stable} 
is indispensable,
as the example below shows.
\begin{example}\label{example:Z12}
Let $f_1$, $f_2$ and $g$ be  general homogeneous polynomials with
$\deg f_1=\deg f_2=1$ and $\deg g=3$.
We put  $B_0:=\{f_1^3 f_2^3+g^2=0\}$.
Then  we have $\Btorus \releqs B_0$.
The $Z$-splitting conic $\Gm=\{f=0\}$ for $\Btorus$  degenerates into the union of two lines
$\{f_1=0\}$ and $\{f_2=0\}$.
Both of them   are  splitting but not pre-$Z$-splitting for $B_0$.
Note that 
 $B_{\torus}$ is lattice-generic, but $B_0$ is not lattice-generic.
\end{example}
\begin{definition}
We call
a pair $(B, \Gm)$ of a  simple sextic $B$ and a $Z$-splitting curve  $\Gm$ for $B$
a \emph{$Z$-splitting pair}.
If $B$ is lattice-generic,
we say that  $(B, \Gm)$ is \emph{lattice-generic}.
\end{definition}
\begin{definition}\label{def:latdataZsplpair}
The \emph{lattice data $\latdataZP(B, \Gm)$  of a $Z$-splitting pair $(B, \Gm)$}
is the extended lattice data
$$
\latdataZP(B, \Gm):=[\EEE_B, h_B, \lat_B, \{[\tlGm^+], [\tlGm^-]\}].
$$
We write $(B, \Gm)   \rellat (B\sprime, \Gm\sprime)$ 
if there exists an isomorphism 
 of  extended  lattice data   
 between $\latdataZP(B, \Gm)$ and $\latdataZP(B\sprime, \Gm\sprime)$.
 The equivalence class of $\rellat$ is called a \emph{lattice type}, and 
 the lattice type  containing a  $Z$-splitting pair $(B, \Gm)$
is denoted by $\lattypeZP (B, \Gm)$.
\end{definition}
By definition,
an isomorphism of lattice data from $\latdataZP(B, \Gm)$ to $\latdataZP(B\sprime, \Gm\sprime)$
is an isomorphism of lattices $\latB\isom\lat_{B\sprime}$
that preserves the polarization class,
the set of exceptional $(-2)$-curves,
and maps the classes of the lifts $\tlGm^\pm$ of $\Gm$  to the classes of the lifts $\tlGm\sp{\prime\pm}$ of $\Gm\sprime$.
\begin{remark}\label{rem:latticegenericZP}
By  Proposition~\ref{prop:stable} and Remark~\ref{rem:lattice_generic}, 
every lattice type $\lattypeZP$ of $Z$-splitting pairs contains a lattice-generic member.
\end{remark}
\section{Main results}\label{sec:mainresults}
\subsection{Classes of lifts of $Z$-splitting curves}\label{subsec:ZZ}
Let $B$ be a simple sextic.
For  $n=1, 2, 3$, we denote by
$$
\ZZZ_n (B):=\set{[\tlGm^+], [\tlGm^-]}{\text{$\Gm$ is a smooth  $Z$-splitting curve of degree $n$}}\;\subset\; \latB.
$$
%
\begin{remark}
In this definition,
the condition that $\Gm$ should be smooth is of course redundant when $n<3$.
For $n=3$,
there may be a $Z$-splitting nodal cubic curve $\Gm$
such that   $\tlGm^+$ and $\tlGm^-$ are $(-2)$-curves on $X_B$, 
but we do not consider  such $Z$-splitting curves.
\end{remark}
The main reason why we treat only smooth $Z$-splitting curves of degree $\le 3$
will be revealed in~Theorem~\ref{thm:FB}.
\par
\medskip
Our first main result is as follows:
\begin{theorem}\label{thm:ZZZ}
Let $B$ and $B\sprime$ be lattice-generic simple sextics such that $B\rellat B\sprime$.
If $\phi:\latB\isom \latt{B\sprime}$ is an isomorphism  of lattice data 
from $\latdata(B)$ to $\latdata(B\sprime)$, then
$\phi$ induces a bijection between $\ZZZ_n (B)$ and $\ZZZ_n (B\sprime)$ for $n=1, 2, 3$.
\end{theorem}
More precisely,
we will give in \S\ref{sec:algorithm} an algorithm 
to calculate the sets $\ZZZ_1 (B)$, $\ZZZ_2 (B)$ and  $\ZZZ_3 (B)$
for a lattice-generic simple sextic $B$ from
the lattice  data $\latdata (B)$.
\par
\medskip
When $n<3$, 
each  element of $\ZZZ_n (B)$ is the class of a unique $(-2)$-curve,
which is a lift of a $Z$-splitting curve of degree $n$.
Hence the cardinality of $\ZZZ_1 (B)$ (resp.~$\ZZZ_2 (B)$) is twice of the number of 
$Z$-splitting lines (resp.~$Z$-splitting conics).
By Theorem~\ref{thm:ZZZ}, we can make the following:
\begin{definition}
For a lattice type $\lattype=\lattype(B)$ of simple sextics,
we define $z_1(\lattype)$ and $z_2(\lattype)$
to be the numbers of $Z$-splitting lines and of $Z$-splitting conics 
for a lattice-generic member $B$ of $\lattype$.
\end{definition}
In the above definition,
the condition that $B$ should be lattice-generic is indispensable.
\begin{example}\label{exampe:z1z2}
The non lattice-generic member $B_0$ 
of the lattice type $\lattype (B_{\torus})=\lattype (B_0)$ in Example~\ref{example:Z12}
has no $Z$-splitting conics,
while $z_2(\lattype (B_{\torus}))=1$.
\end{example}
The usefulness of 
the notion of $Z$-splitting curves in the study of
lattice Zariski $k$-ples comes from  the following:
\begin{theorem}\label{thm:LZkplets}
Let $\lattype$ and $\lattype\sprime$ be 
lattice types  of simple sextics
in the same  configuration type.
If $z_1(\lattype)=z_1(\lattype\sprime)$ and $z_2(\lattype)=z_2(\lattype\sprime)$,
then $\lattype=\lattype\sprime$.
Namely, the lattice types in any lattice Zariski $k$-ple are
distinguished by the numbers $z_1(\lattype)$ and $z_2(\lattype)$.
\end{theorem}
The set $\ZZZ_3 (B)$ is in two-to-one correspondence with a set of
one-dimensional families of  $Z$-splitting cubic curves.
\begin{proposition}\label{prop:cubic}
Let $\tlE$ be an effective divisor on $X_B$.
We have 
 $[\tlE]\in  \ZZZ_3 (B)$ if and only if 
$|\tlE|$ is an elliptic pencil on $X_B$ whose general member is a lift of a $Z$-splitting cubic curve.
\end{proposition}
\begin{proof}
Let $\tlE$ be a lift of a smooth $Z$-splitting cubic curve $E$.
Then $\tlE$ is smooth of genus $1$,
and hence $|\tlE|$ is an elliptic pencil.
Conversely,
if $|\tlE|$ is an elliptic pencil on $X_B$ whose general member $\tlE$ is a lift of a $Z$-splitting cubic curve $E$,
then $E$ must be smooth because $E$ is birational to $\tlE$ and hence of genus $1$.
\end{proof}
\subsection{Classification of $Z$-splitting curves of degree $\le 2$}
Next we give a  classification of lattice types of $Z$-splitting pairs $(B, \Gm)$ with $\deg \Gm\le 2$.
The numbers of lattice types $\lattype$ of simple sextics  with $z_1(\lattype)>0$
or $z_2(\lattype)>0$
are  given in Table~\ref{table:z1z2numbs}.
%
If $\mu_B<12$,  then $B$ has no $Z$-splitting curves of degree $\le 2$.
(Remark that there are lattice types $\lambda$ for which   both $z_1(\lattype)>0$ and $z_2(\lattype)>0$ hold.
Such lattice types are counted twice in Table~\ref{table:z1z2numbs}.)
\begin{table}
$$
\begin{array}{c|cccccccc|c}
\mu_B &12 &13 &14 &15 &16 &17 &18 &19 &\rm{total}\mystrutd{4pt}\\ 
 \hline 
\rm{lines}&0 &0 &0 &1 &2 &7 &13 &18 &41 \mystruthd{12pt}{4pt}\\ 
\rm{conics}&1 &2 &7 &18 &47 &86 &108 &55 &324\\ 
\end{array}
$$

\caption{Numbers of lattice types with $Z$-splitting lines or conics}\label{table:z1z2numbs}
\end{table}
\par
\medskip
The entire classification  table  is
too huge to be presented in a paper.
In order to state our classification in a concise way,
we introduce the notion of
\emph{specialization} of lattice types.
\begin{definition}
Let $\lattype_0$ and $\lattype$ be lattice types of simple sextics.
 We say that $\lattype_0$ is a \emph{specialization} of  $\lattype$
 if  there exists an analytic  family $f:\BBB\to \Delta$ of simple sextics $f\inv (t)=B_t$  
parameterized by a unit disc $\Delta\subset \C$, where $\BBB$ is a surface in $\Pt\times\Delta$ and $f$ is a projection, 
such that the central fiber $B_0$ is a member of $\lattype_0$ and
the other  fibers $B_t$ $(t\ne 0)$ are members of $\lattype$.
\end{definition}
\begin{definition}
Let $\lattypeZP_0$ and $\lattypeZP$ be lattice types of  $Z$-splitting pairs.
 We say that $\lattypeZP_0$ is a \emph{specialization} of  $\lattypeZP$
if  there exists  an analytic  family 
$f : \PPP \to \Delta$
of $Z$-splitting pairs $f\inv (t)=(B_t, \Gm_t)$ 
such that  the central fiber  $f\inv (0)$ is a  member of
$\lattypeZP_0$ and the other  fibers  
$f\inv (t)$ $(t\ne 0)$  are 
 members of $\lattypeZP$.
\end{definition}
We  give the list of lattice types of $Z$-splitting pairs that generate
all other lattice types  by specialization.
It turns out that the lineages of lattice types 
via specialization are classified by the \emph{class-order} defined below.
\begin{definition}
 The \emph{class-order} of a  $Z$-splitting pair $(B, \Gm)$ 
 (or of a lattice type $\lattypeZP(B, \Gm)$ of $Z$-splitting pairs)
 is the order of $[\tlGmplus]$ in $G_B=\latB/\blatB$.
\end{definition}
%

In the following,
the index $\tau$ in lattice types $\lattype_{\alpha, \tau}$
takes symbolic values
$n$, $l$ or $c$,
which stand for ``\emph{none}", ``\emph{line}" and ``\emph{conic}",
respectively.
\par
\medskip
The classification of $Z$-splitting lines is as follows.
\begin{definition}\label{def:Zlinesconfig}
For $\alpha\in \{\AAAA,\BBBB,\CCCC,\DDDD\}$,
let $\configtype_\alpha$ be the configuration type 
given in the following table:
$$
\renewcommand{\arraystretch}{1.2}
\begin{array}{cllll}
\alpha  & R_B &\degs && \textrm{} \\
\hline
\AAAA & 3A_5 & [3,3]&& (\textrm{the cubics are smooth})\\
\BBBB & A_3+2A_7 & [2, 4]&&  (\textrm{the quartic has $A_3$})\\
\CCCC & 2A_4+A_9 & [1, 5]&& (\textrm{the quintic has $2 A_4$}) \\
\DDDD & A_3+A_5+A_{11} & [2, 4]&&  (\textrm{the quartic has $A_5$})
\end{array}
$$
\end{definition}
\begin{proposition}\label{prop:Zlineslattype1}
Let $\alpha$ be one of $\AAAA,\BBBB,\CCCC,\DDDD$.

{\rm (1)} The lattice types $\lattype_{\alpha, \tau}$ in 
the configuration type $\configtype_\alpha$ 
are given in Table~\ref{table:lattypeZPorglin}.
The invariants $z_1(\lattype_{\alpha, \tau})$,
$z_2(\lattype_{\alpha, \tau})$,
$G_B$ and $F_B$ of these lattice-types are also given in 
this table.

{\rm (2)} Let $B$ be a lattice-generic member of 
$\lattype_{\alpha, l}$,
so that there exists a unique $Z$-splitting line $\Gm$ for $B$.
Then $\Gm$ passes through the three singular points of $B$,
and the cyclic group $G_B$ is generated by $[\tlGmplus]$.
\end{proposition}
\begin{table}
$$
\renewcommand{\arraystretch}{1.2}
\begin{array}{ccllccccccc}
\alpha && R_B &\degs&& \tau &  z_1 & z_2 && \;\;\;\;G_B\;\;\;\;  & F_B \\
\hline
\AAAA && 3A_5 &[3,3]&&l&1 & 0 && \Z/6\Z & \Z/3\Z  \\ 
  &&      &&&n&0 & 0 && \Z/2\Z & 0\\ 
\hline 
\BBBB && A_3+2A_7 &[2,4]&& l &1 & 0 && \Z/8\Z & \Z/4\Z \\ 
  &&          &&& c &0 & 1 && \Z/4\Z &  \Z/2\Z \\
  &&          &&& n &0 & 0 && \Z/2\Z & 0 \\ 
\hline 
\CCCC && 2A_4+A_9  &[1,5]&& l &1 & 0 && \Z/10\Z& \Z/5\Z\\ 
  &&           &&& n &0 & 0 && \Z/2\Z & 0  \\ 
\hline 
\DDDD && A_3+A_5+A_{11} &[2,4]&& l  &1 & 1 &&  \Z/12\Z & \Z/6\Z\\ 
\end{array}
$$
\vskip 2mm
\caption{Lattice types $\lattype_{\alpha, \tau}$ in $\configtype_\alpha$  for $\alpha\in \{\AAAA,\BBBB,\CCCC,\DDDD\}$}\label{table:lattypeZPorglin}
\end{table}
\begin{definition}\label{def:Zlineslattype2}
Let $B$ and $\Gm$ be as in Proposition~\ref{prop:Zlineslattype1}~(2).
We put
$$
\lattypeZP_{lin, d}:=\lattypeZP(B, \Gm),
$$
where $d$ is the order of $G_B$;
that is,  $d=6,8,10,12$
according to $\alpha=\AAAA,\BBBB,\CCCC,\DDDD$.
\end{definition}
These lattice types $\lattypeZP_{lin, d}$ are the originators of the lineages of
lattice types of $Z$-splitting lines.
\begin{theorem}\label{thm:Zlines}
Let $(B, \Gm)$ be  a  $Z$-splitting pair with $\deg \Gm=1$.
Then the class-order $d$ of  $\lattypeZP(B, \Gm)$ is  $6, 8, 10$ or $ 12$,
and $\lattypeZP(B, \Gm)$ is a  specialization of the lattice type $\lattypeZP_{lin, d}$.
\end{theorem}
The classification of $Z$-splitting conics is as follows.
\begin{definition}\label{def:Zconicsconfig}
For $\alpha\in \{\aaaa,\bbbb,\cccc,\dddd,\eeee,\ffff\}$,
let $\configtype_\alpha$ be the configuration type 
given in the following table:
$$
\renewcommand{\arraystretch}{1.15}
\begin{array}{cllll}
\alpha  & R_B &\degs && \textrm{} \\
\hline
\aaaa & 6A_2 &[6]&& \\ 
\bbbb & 2A_1+4A_3 &[2,4]&&  (\textrm{the quartic has $2A_1$}) \\
\cccc & 4A_4 &[6]&&  \\
\dddd & 2A_1+2A_2+2A_5 &[2,4] && (\textrm{the quartic has $2A_2$}) \\
\eeee & 3A_6 &[6]&& \\
\ffff & A_1+A_3+2A_7 &[2,4]&& (\textrm{the quartic has $A_1+A_3$})\\
\end{array}
$$
\end{definition}
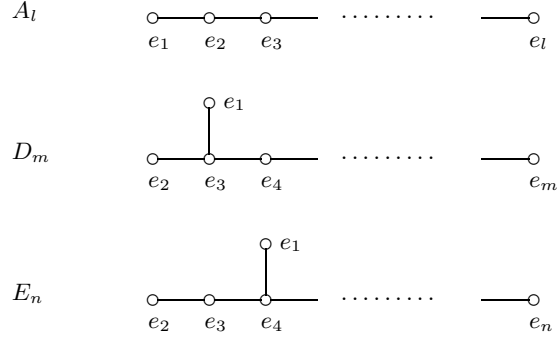
\begin{figure} 
\def\ha{40}
\def\hav{37}
\def\hd{25}
\def\hdv{22}
\def\he{10}
\def\hev{7}
\setlength{\unitlength}{1.25mm}
\centerline{
{\small
\begin{picture}(100, 37)(-20, 7)
\put(-5, \ha){$A\sb l$}
\put(10, \ha){\circle{1}}
\put(9.5, \hav){$e\sb 1$}
\put(10.5, \ha){\line(5, 0){5}}
\put(16, \ha){\circle{1}}
\put(15.5, \hav){$e\sb 2$}
\put(16.5, \ha){\line(5, 0){5}}
\put(22, \ha){\circle{1}}
\put(21.5, \hav){$e\sb 3$}
\put(22.5, \ha){\line(5, 0){5}}
\put(30, \ha){$\dots\dots\dots$}
\put(45, \ha){\line(5, 0){5}}
\put(50.5, \ha){\circle{1}}
\put(50, \hav){$e\sb {l}$}
\put(-5, \hd){$D\sb m$}
\put(10, \hd){\circle{1}}
\put(9.5, \hdv){$e\sb 2$}
\put(10.5, \hd){\line(5, 0){5}}
\put(16, 31){\circle{1}}
\put(17.5, 30.5){$e\sb 1$}
\put(16, 25.5){\line(0,1){5}}
\put(16, \hd){\circle{1}}
\put(15.5, \hdv){$e\sb 3$}
\put(16.5, \hd){\line(5, 0){5}}
\put(22, \hd){\circle{1}}
\put(21.5, \hdv){$e\sb 4$}
\put(22.5, \hd){\line(5, 0){5}}
\put(30, \hd){$\dots\dots\dots$}
\put(45, \hd){\line(5, 0){5}}
\put(50.5, \hd){\circle{1}}
\put(50, \hdv){$e\sb {m}$}
\put(-5, \he){$E\sb n$}
\put(10, \he){\circle{1}}
\put(9.5, \hev){$e\sb 2$}
\put(10.5, \he){\line(5, 0){5}}
\put(16, \he){\circle{1}}
\put(15.5, \hev){$e\sb 3$}
\put(22, 16){\circle{1}}
\put(23.5, 15.5){$e\sb 1$}
\put(22, 10.5){\line(0,1){5}}
\put(16.5, \he){\line(5, 0){5}}
\put(22, \he){\circle{1}}
\put(21.5, \hev){$e\sb 4$}
\put(22.5, \he){\line(5, 0){5}}
\put(30, \he){$\dots\dots\dots$}
\put(45, \he){\line(5, 0){5}}
\put(50.5, \he){\circle{1}}
\put(50, \hev){$e\sb {n}$}
\end{picture}
}
}
\vskip 5pt
\caption{Dynkin diagram}\label{figure:dynkin}
\end{figure}
\begin{definition}\label{def:tau}
Let $P$ be a singular point of  $B$, and 
let $e_1, \dots, e_r$ be the exceptional $(-2)$-curves
on $\XB$ over $P$ indexed in such a way that
the dual graph  is given in Figure~\ref{figure:dynkin}.
Let $\tlGmplus$ be a lift of a smooth splitting curve $\Gm$. 
Suppose that  $P\in \Gm$.
Since $\Gm$ is smooth and splitting, 
 there exists a unique  $e_j$  among $e_1, \dots, e_r$
that intersects $\tlGmplus$. 
(See Lemma~\ref{lem:vsmooth}.)
We put $\tau_P (\tlGmplus):=j$.
If $P\notin \Gm$,
we put $\tau_P (\tlGmplus):=0$ and $\tau_P (\tlGmminus):=0$.
\end{definition}
\begin{proposition}\label{prop:Zconicslattype1}
Let $\alpha$ be one of $\aaaa,\bbbb,\cccc,\dddd,\eeee,\ffff$.

{\rm (1)}
The lattice types $\lattype_{\alpha, \tau}$ in the configuration type $\configtype_\alpha$ are given
in Table~\ref{table:lattypeZPorgcon},
together with
the invariants $z_1(\lattype_{\alpha, \tau})$,
$z_2(\lattype_{\alpha, \tau})$,
$G_B$ and $F_B$.

{\rm (2)}
Let $B$ be a lattice-generic member of $\lattype_{\alpha, c}$.
Then the $Z$-splitting conics $\Gm$ for $B$ are given  in Table~\ref{table:Zconics},
where $\ord  $ is the class-order of  $(B, \Gm)$,
and $\tau_P(\tlGmplus)$ is described under  an appropriate choice of 
numbering of the exceptional $(-2)$-curves and the lift of $\Gm$.
\begin{table}
$$
\renewcommand{\arraystretch}{1.25}
\begin{array}{ccllcclccccc}
\alpha && R_B &\degs && \tau &&  z_1 & z_2 && \;\;\;\;G_{B}\;\;\;\; & F_{B}   \\
\hline
\aaaa && 6A_2  &[6]& & c&&0 & 1 && \Z/3\Z & \Z/3\Z  \\ 
&&         &&& n&&0 & 0 && 0 &0  \\ 
\hline 
\bbbb && 2A_1+4A_3 &[2,4]&& c&&0 & 1 && \Z/4\Z &\Z/2\Z\\ 
&&             &&& n&&0 & 0 && \Z/2\Z & 0 \\ 
\hline 
\cccc && 4A_4 &[6]&& c&&0 & 2 && \Z/5\Z & \Z/5\Z\\ 
&&       &&& n&&0 & 0 && 0 &0 \\ 
\hline 
\dddd && 2A_1+2A_2+2A_5 &[2,4]&& c&&0 & 2 && \Z/6\Z &\Z/3Z \\ 
 &&                &&& n&&0 & 0 && \Z/2\Z &0 \\ 
\hline 
\eeee && 3A_6  &[6]&& c&&0 & 3 && \Z/7\Z&\Z/7\Z \\ 
&&         &&& n&&0 & 0 && 0 &0    \\ 
\hline 
\ffff && A_1+A_3+2A_7 &[2,4]&& c&&0 & 3 && \Z/8\Z &\Z/4\Z \\ 
&&                &&& l&&1 & 0 && \Z/8\Z &\Z/4\Z \\
&&                &&& n&&0 & 0 && \Z/4\Z &\Z/2\Z \\ 
\end{array}
$$
\vskip 2mm
\caption{Lattice types $\lattype_{\alpha, \tau}$  in $\configtype_\alpha$  for $\alpha\in \{\aaaa,\bbbb, \dots, \ffff\}$}\label{table:lattypeZPorgcon}
\end{table}
\end{proposition}
%
%
%
%
\begin{table}
\newcommand{\taubox}[1]{\hbox to 11mm {\hss$#1$\hss}}
$$
\renewcommand{\arraystretch}{1.2}
\begin{array}{lccl}
\alpha & \Gm  & \ord  &  \hss\tau_P(\tlGmplus)\hss \\
 \hline
\aaaa &&& \taubox{A_2} \taubox{A_2}\taubox{A_2}\taubox{A_2}\taubox{A_2}\taubox{A_2} \\
 \cline{4-4}
&\Gm & 3 & \taubox{1} \taubox{1}\taubox{1}\taubox{1}\taubox{1}\taubox{1}\\
\hline
\bbbb &&& \taubox{A_1} \taubox{A_1}\taubox{A_3}\taubox{A_3}\taubox{A_3}\taubox{A_3} \\
 \cline{4-4}
&\Gm  & 4 & \taubox{1} \taubox{1}\taubox{1}\taubox{1}\taubox{1}\taubox{1}\\
\hline
\cccc &  &  & \taubox{A_4} \taubox{A_4}\taubox{A_4}\taubox{A_4} \\
 \cline{4-4}
&\Gm_1&5& \taubox{1} \taubox{1}\taubox{2}\taubox{2}\\
 \cline{4-4}
 &\Gm_2&5& \taubox{2} \taubox{2}\taubox{4}\taubox{4}\\
 \hline
\dddd &  &  & \taubox{A_1} \taubox{A_1}\taubox{A_2}\taubox{A_2}\taubox{A_5}\taubox{A_5} \\
 \cline{4-4}
  &\Gm_1&6&  \taubox{1} \taubox{1}\taubox{2}\taubox{2}\taubox{1} \taubox{1}\\
 \cline{4-4}
&\Gm_2&3& \taubox{0} \taubox{0}\taubox{1}\taubox{1}\taubox{2} \taubox{2}\\
 \hline
\eeee &  &  & \taubox{A_6} \taubox{A_6}\taubox{A_6} \\
 \cline{4-4}
&\Gm_1&7& \taubox{1} \taubox{2}\taubox{3}\\
 \cline{4-4}
 &\Gm_2&7&  \taubox{2} \taubox{4}\taubox{6}\\
  \cline{4-4}
  &\Gm_3&7&  \taubox{3} \taubox{6}\taubox{2}\\
    \hline
\ffff &  &  & \taubox{A_1} \taubox{A_3}\taubox{A_7}\taubox{A_7} \\
 \cline{4-4}
&\Gm_1&8& \taubox{1} \taubox{1}\taubox{1}\taubox{5}\\
 \cline{4-4}
 &\Gm_2&4&  \taubox{0} \taubox{2}\taubox{2}\taubox{2}\\
  \cline{4-4}
  &\Gm_3&8& \taubox{1} \taubox{3}\taubox{3}\taubox{7}\\
   \hline
  \end{array}
$$
\vskip 5pt
\caption{$Z$-splitting conics of  $\lattype_{\alpha, c}$ for $\alpha\in \{\aaaa,\bbbb, \dots, \ffff\}$}\label{table:Zconics}
\end{table}
\begin{definition}\label{def:Zconicslattype2}
Let $B$  be as in Proposition~\ref{prop:Zconicslattype1}~(2),
and let $\Gm$ be a $Z$-splitting conic for $B$ such that $[\tlGmplus]$ generates $G_B$.
We put
$$
\lattypeZP_{con, d}:=\lattypeZP(B, \Gm),
$$
where $d$ is the order of $G_B$.
\end{definition}
\begin{remark}\label{rem:Zconicslattype2}
For $d=5,7,8$, 
the lattice type $\lattypeZP_{con, d}=\lattypeZP(B, \Gm)$ does not depend on the choice of 
$\Gm$ as long as  $[\tlGmplus]$ generates $G_B$.
\end{remark}
These lattice types $\lattypeZP_{con, d}$ are the originators of the lineages of
lattice types of $Z$-splitting conics.
\begin{theorem}\label{thm:Zconics}
Let $(B, \Gm)$ be a $Z$-splitting pair with $\deg \Gm=2$.
Then the class-order $d$ of  $\lattypeZP(B, \Gm)$ is  $3,4,5,6,7$ or $ 8$,
and $\lattypeZP(B, \Gm)$ is a  specialization of the lattice type $\lattypeZP_{con, d}$.
\end{theorem}
\begin{remark}
The simple sextics  in $\lattypeZP_{con, 3}$ 
are the classical torus sextics,
which have been studied in details by many authors~(for example, see~\cite{MR1948673}).
The simple sextics  in $\lattypeZP_{con, 5}$ 
are studied by Degtyarev in~\cite{MR2439628} and~\cite{degtyarev-2007}.
The simple sextics  in $\lattypeZP_{con, 7}$ 
are studied by Degtyarev in~\cite{MR2439628} and by Degtyarev-Oka in~\cite{degtyarev-oka-2007}.
\end{remark}
%
%
%
\subsection{Generators of $F_B$ and $Z$-splitting cubic curves}\label{subsec:genFB}
%
%
%
\begin{theorem}\label{thm:FB}
\setrmkakko
Let $B$ be a lattice-generic member  of a lattice type $\lattype=\lattype (B)$.

\rmkakko 
The finite abelian group  $F_B=\latB/\Theta_B$ is  generated by the classes of 
lifts of smooth $Z$-splitting curves of degree $\le 3$;
that is, we have
\begin{equation}\label{eq:gen123}
\latB=\bBlatB+\gen{\ZZZ_1(B)}+\gen{ \ZZZ_2(B)}+\gen{ \ZZZ_3(B)}.
\end{equation}

\rmkakko If $z_1(\lattype)>0$ or $z_2(\lattype)>0$,
then $F_B$ is non-trivial
and is generated by the classes of 
lifts of $Z$-splitting curves of degree $\le 2$.
\end{theorem}
The generators $\gen{ \ZZZ_3(B)}$ are indispensable in~\eqref{eq:gen123},
as the following example $\lattype_{QC, n}$ shows.
\begin{proposition}\label{prop:sp}
Let $\gamma_{\QC}$ be 
the configuration type  of simple sextics $B=Q+ C$
with $\degs B=[2,4]$, $R_B=3A_1+4A_3$ and the quartic curve $Q$ having $3A_1$.
\setrmkakko

\rmkakko
The configuration type $\gamma_{\QC}$
contains exactly two lattice types $\lattype_{\QC, n}$ and $\lattype_{\QC, c}$,
which  are distinguished by the following:
 $$
 z_1(\lattype_{\QC, c})=0, \;\; z_2(\lattype_{\QC, c})=1,
 \qquad
  z_1(\lattype_{\QC, n})=0, \;\; z_2(\lattype_{\QC, n})=0.
 $$
 These lattice types have isomorphic $G_B$ and $F_B$;
 for a  member $B$ of  $\gamma_{\QC}$,
 $G_B$ is cyclic of order $4$ and $F_B$ is of order $2$.
 
\rmkakko
Let $B=Q+C$ be a lattice-generic member of $\lattype_{\QC, c}$,
and let  $\Gm$  be the unique $Z$-splitting  conic for $B$.
Then  $G_B$ is generated by $[\tlGmplus]$.

\rmkakko
Let $B\sprime=Q\sprime+C\sprime$ be a lattice-generic member of $\lattype_{\QC, n}$,
so that $\ZZZ_1(B\sprime)=\ZZZ_2(B\sprime)=\emptyset$.
Then $\ZZZ_3(B\sprime)$ consists of two elements $[\tlE\sp+]$ and $[\tlE\sp-]$,
and  $G_{B\sprime}$ is generated by $[\tlE\sp +]$.
Let $E$ be the image of a general member of the elliptic pencil $|\tlE\sp+|$,
which is a smooth $Z$-splitting cubic curve.
Then $E$ passes through every point of $\Sing B\sprime$
and is tangent to each of $Q\sprime$ and $C\sprime$.
\end{proposition}
We need 
 $Z$-splitting cubic curve to generate $F_{B\sprime}\ne 0$.
We put
$$
\lattypeZP_{\QC, n}:=\lattypeZP(B\sprime, E),
$$
where $(B\sprime, E)$ is the $Z$-splitting pair in Proposition~\ref{prop:sp}~(3).
The lattice type $\lattypeZP_{\QC, n}$ is the ancestor of all 
lattice types for which  we need $Z$-splitting cubic curves to generate $F_B$.
\begin{theorem}\label{thm:sp}
\setrmkakko
Let  $\lattype_0$ be a lattice type of simple sextics with a lattice-generic member $B_0$.
Suppose that  $z_1(\lattype_0)=0$ and $z_2(\lattype_0)=0$  but
$F_{B_0}\ne 0$.

\rmkakko
The set $\ZZZ_3(B_0)$ consists of two elements $[\tlE_0\sp+]$ and $[\tlE_0\sp-]$, and 
  $G_{B_0}$ is cyclic of order $4$  generated by $[\tlE_{0}\sp +]$.

\rmkakko
Let $E_{0}$ be the image of a general member of the elliptic pencil $|\tlE_0\sp+|$.
Then  the lattice type  $\lattypeZP(B_0, E_0)$ is a specialization of  the lattice type $\lattypeZP_{\QC, n}$ defined above.
\end{theorem}
\section{Classification of lattice types of simple sextics}\label{sec:Yang}
\subsection{Fundamental system of roots}\label{subsec:fundamental-system-of-roots}
Let $\rlat$ be an even negative-definite  lattice,
and 
let  $D_\rlat$ be the set of roots in $\rlat$.
We denote by $\spzero  \Hom(L, \R)$ the space of all linear forms
$t: L\to \R$ such that 
$t(d)\ne 0$ holds
for any $d\in D_\rlat$.
For $t\in \spzero  \Hom(L, \R)$,
we put
$$
(D_\rlat)_t^+:=\set{d\in D_\rlat}{t(d)>0}.
$$
An element $d\in (D_\rlat)_t^+$
is said to be  \emph{decomposable}
if there exist  $d_1, d_2\in (D_\rlat)_t^+$ such that $d=d_1+d_2$;
otherwise,  we say that  $d$  is \emph{indecomposable}.
The proof of the following well-known fact is found, for example, in Ebeling~\cite[Proposition~1.4]{MR1938666}.
\begin{proposition}\label{prop:t-roots}
The set $F_t$ of indecomposable elements in $(D_\rlat)_t^+$ is 
a fundamental system of roots in $\rlat$.
Conversely, if $F$ is a fundamental system of roots in $\rlat$,
then there exists a linear form $t\sprime \in \spzero  \Hom(L, \R)$
such that $F$ is equal to the set  $F_{t\sprime}$ 
of  indecomposable elements in $(D_\rlat)_{t\sprime}^+$.
\end{proposition}
We  call  $F_t$
the  \emph{fundamental system of roots associated with $t: \rlat \to \R$}.
\begin{corollary}\label{cor:geom_fund_roots}
There exists a one-to-one correspondence 
between the set of fundamental systems of roots in $\rlat$
and the set of connected components of $\spzero \Hom(L, \R)$.
\end{corollary}
\begin{remark}
A fundamental system  of roots $F$ in $\rlat$ is associated with $t\in \spzero \Hom(L, \R)$
if and only if $\intnum{t}{d}>0$ holds for any $d\in F$.
\end{remark}
\subsection{The K\"ahler cone and polarizations of a $K3$ surface}\label{subsec:Kaehler}
Let $X$ be a  $K3$ surface, and 
let $\per_X$ be a basis of $H^{2,0}(X)$.
We  put
\begin{eqnarray*}
H_{X} &:=& \set{x\in H^2(X, \R)}{\intnum{x}{\per_X}=0},\\
D_{X}&:=& \set{d\in \NS(X)}{d^2=-2}, \\
\Gamma_{X} &:=& \set{x\in H_{X}}{x^2>0}, \\
\zeroGamma_{X} &:=& \set{x\in \Gamma_{X} }{ \intnum{x}{d}\ne 0\;\;\textrm{for all}\;\; d\in D_{X}}.
\end{eqnarray*}
We have $H_{X}=H^2(X, \R)\cap H^{1,1}(X)$ and $\NS(X)=H^2(X, \Z)\cap H_X$.
We also have 
$$
\Gamma_X = \Gamma_X^+\sqcup (-\Gamma_X^+) \quad({\rm disjoint}),
$$
where  $\Gamma_X^+$ is the connected component of $\Gamma_X$
that contains  a K\"ahler class of $X$.
\begin{definition}
The \emph{K\"ahler cone} $\Kahler_X$ of $X$
is the set of vectors 
$\kappa\in H_X$
satisfying  $\intnum{D}{\kappa}>0$ 
for any effective divisor $D$  on $X$.
\end{definition}
%
%
Every K\"ahler class of $X$ is contained in $\Kahler_X$.
Conversely, 
as a corollary of Theorem~\ref{thm:refined} below, 
we see that
every vector in $\Kahler_X$ is a K\"ahler class on $X$.
%
\par
\medskip
The following proposition is an immediate consequence of the definition.
\begin{proposition}\label{prop:nefcone}
A vector $v\in \NS(X)$ is nef if and only if $v$ is contained in the closure  of 
the K\"ahler cone $\Kahler_X$ in $H_X$. 
\end{proposition}
We set
$$
\Delta_X:=\set{d\in D_X}{\textrm{$d$ is effective}}.
$$
By  Riemann-Roch theorem,
we see that $D_X$ is a disjoint union of $\Delta_X$ and $-\Delta_X$.
For $d\in D_X$, we put
$$
d\sperp:=\set{x\in H_X}{\intnum{x}{d}=0},
$$
and call $d\sperp$ the \emph{wall} associated with $d\in D_X$.
The family of walls $\shortset{d\sperp}{d\in D_X}$
is locally finite in the cone $\Gamma_X$,
and partitions $\Gamma_X$ into the connected components of $\zeroGamma_X$. 
The following  proposition is  well-known.
(See, for example,~\cite[Corollary~3.9 in Chap.~VIII]{MR2030225}).
\begin{proposition}\label{prop:Kcone}
The K\"ahler cone 
$\Kahler_X\subset H_X$ is the unique connected component of $\Gamma_X^+\cap \zeroGamma_X$
such that 
$\intnum{x}{d}>0$ holds for every $d\in \Delta_X$ and every $x\in \Kahler_X$.
\end{proposition}
A line bundle $\pol$ on  $X$ is called a \emph{polarization}
if $\pol$ is nef, $\pol^2>0$,   and 
the complete linear system  $|\pol|$ has no fixed components.
If $\pol$  is a polarization, then $|\pol|$ has no base points by~\cite[Corollary 3.2]{MR0364263},
and hence defines a morphism
$$
\Phi_{|\pol|}: X\to \P^N,
$$
where $N=\dim |\pol|$.
\begin{proposition}\label{prop:polarization}
A vector  $v\in \NS(X)$ is the  class of a polarization
if and only if $v^2>0$, $v$ is nef, and the set
$\shortset{x\in \NS(X)}{\intnum{v}{x}=1, x^2=0}$
is empty.
\end{proposition}
\begin{proof}
See
Nikulin~\cite[Proposition 0.1]{MR1260944},  and 
the argument in the proof of (4)$\Rightarrow$(1)
in Urabe~\cite[Proposition 1.7]{MR1101859}.
\end{proof}
Let $\pol$ be a polarization on $X$.
The orthogonal complement 
$[\pol]\sperp$
of $\gen{[\pol]}$ in $\NS(X)$ is  negative-definite by Hodge index theorem.
Then we can easily prove the following.
(See~\cite[Proposition 2.4]{MR2369942}.)
\begin{proposition}\label{prop:geomroots}
The set of
 classes of $(-2)$-curves 
that are contracted by    $\Phi_{|\pol|}$
 is equal to the fundamental system of roots
in $[\pol]\sperp$
associated with the linear form 
$t_\kappa: [\pol]\sperp \to \R$
given by $t_\kappa(v):=\intnum{v}{\kappa}$, 
where  $\kappa$ is  a vector in the K\"ahler cone $\Kahler_X$.
\end{proposition}
%
%
\begin{corollary}\label{cor:walls}
Let $U\subset H_X$ be a sufficiently small open  ball with the center $[\pol]$.
Then  $U\cap \Kahler_X$ is an open cone with the vertex $[\pol]$ and with the faces  being  the  walls $d\sperp$,
where $d$ are the $(-2)$-curves contracted by  $\Phi_{|\pol|}$.
\end{corollary}
\subsection{Lattice types  of simple sextics}\label{sec:lattice-types-of-simple-sextics}
We denote by $\Klat$ the $K3$ lattice,
that is, an even unimodular lattice of signature $(3, 19)$,
which is unique up to isomorphisms.
We put
$$
\Omega_{\Klat}:=\set{[\per]\in \P_*(\Klat\otimes\C)}{\intnum{\per}{\per}=0, \;\intnum{\per}{\bar\per}>0},
$$
which is a complex manifold of dimension $20$ with two connected components.
A \emph{marked $K3$ surface} is a pair $(X, \phi)$
of a $K3$ surface $X$ and an isomorphism  $\phi: H^2(X, \Z)\isom \Klat$ of lattices.
There exists a universal family 
$$
(\pi_1: \XXX_1\to\MMM_1, \Phi_1)
$$
of marked $K3$ surfaces
over a \emph{non-Hausdorff} smooth complex manifold $\MMM_1$ of dimension $20$,
where $\Phi_1$ is 
an isomorphism $R^2\pi_{1*}\Z\cong \MMM_1\times \Klat$ of locally constant systems of lattices over $\MMM_1$.
(See~\cite[\S12 of Chap.~VIII]{MR2030225} or~\cite{MR785231}.)
For $t\in \MMM_1$,
we have  a point
$$
\tau_1 (t):=[\phi_{t}(\per_{X_t})]\;\;\in \;\;\Omega_{\Klat},
$$
where $(X_t, \phi_t)$ is the  marked $K3$ surface corresponding to $t$,
and $\per_{X_t}$ is a basis of $H^{2,0}(X_t)$.
We call $\tau_1 (t)$ the \emph{period point of $(X_t, \phi_t)$}.
It is well-known  that the \emph{period map} 
$$
\tau_1 \;:\; \MMM_1\;\to\; \Omega_{\Klat}
$$
 is   holomorphic and surjective.
 (See~\cite[\S12 of Chap.~VIII]{MR2030225} or~\cite{MR785231}.)
 %
 %
%
%
\par
\medskip
Yang~\cite{MR1387816} presented an algorithm to classify  all  lattice data 
that can be  realized as lattice data  of  simple sextics.
His method  is based on the following proposition, 
which was proved by  the surjectivity of $\tau_1$ and 
Propositions~\ref{prop:polarization} and~\ref{prop:geomroots}.
\begin{proposition}[Urabe~\cite{MR1101859, MR1000608}]\label{prop:urabe}
Lattice data $[\EEE, h, \lat]$ is isomorphic to  lattice data of simple sextics
if and only if $[\EEE, h, \lat]$ satisfies the following:
\begin{itemize}
\item[(i)] the lattice $\lat$ can be  embedded primitively in $\Klat$,
\item[(ii)]  $\shortset{x\in \lat}{\intnum{x}{h}=0, \;x^2=-2}\;=\;\shortset{x\in \gen{\EEE}}{x^2=-2}$, and
\item[(iii)]  $\shortset{x\in \lat}{\intnum{x}{h}=1, \;x^2=0}\;=\;\emptyset$.
\end{itemize}
\end{proposition}
\begin{computation}\label{comp:yang}
Let $R$ be an $ADE$-type of rank $\le 19$.
We determine all lattice data  of simple sextics $B$ with $R_B=R$.
We put $\blat:=\gen{h}\oplus \gen{\EEE}$,
where $h^2=2$ and $\EEE$ is the fundamental system of roots of type $R$.
We then calculate the \emph{discriminant form} of $\blat$.
(See~\cite[\S1]{MR525944} for the definition of the discriminant form of an even lattice.)
We then make the complete list of isotropic subgroups $H$ of the discriminant form of $\blat$.

For each isotropic subgroup $H$, 
we calculate the even overlattice $\lat(H)$ of $\blat$ corresponding to $H$ by~\cite[Proposition 1.4.1]{MR525944}.
We then determine whether or not $\lat=\lat(H)$ satisfies the conditions (ii) and (iii) in Proposition~\ref{prop:urabe}
by the method described in~\cite[\S4]{MR2036331}, and then determine 
 whether  or not  $\lat(H)$ can be embedded primitively into $\Klat$   by means of~\cite[Theorem~1.12.1]{MR525944} or 
 by the method of $p$-excess due to 
Conway-Sloane~\cite[Chap.~15]{MR1662447} described in~\cite[\S3]{MR2369942}. 
(See also~\cite[Chapters 8 and 9]{MR522835}.)

We conclude that $[\EEE, h, \lat(H)]$ is realized as lattice data  of simple sextics $B$ with $R_B=R$
if and only if $\lat(H)$ satisfies the conditions  in Proposition~\ref{prop:urabe}.
\end{computation}
More precisely,
the family of simple sextics $B$ with $\latdata(B)\cong [\EEE, h, \lat]$
is described as follows.
Suppose that  lattice data $[\EEE, h, \lat]$  satisfies the conditions (i), (ii) and (iii)
in Proposition~\ref{prop:urabe}.
We choose a primitive  embedding
$$
\psi: \lat\inj \Klat,
$$
and consider $\lat$ as a primitive sublattice of $\Klat$.
In particular, 
we have $\EEE\subset \Klat$ and $h\in \Klat$.
\begin{remark}
The  primitive embedding of $\lat$ in $\Klat$ is not unique in general.
In fact, by choosing different primitive embeddings of $\lat$ in $\Klat$,
we often obtain distinct connected components of the equisingular family~(see Degtyarev~\cite{MR2357681}).
More strongly, 
we have obtained  examples of pair of simple sextics $B_1$ and $B_2$ such that $B_1\rellat B_2$
but $B_1\nrelemb B_2$ 
by considering different primitive embeddings of $\lat$ (see~\cite{withArima},~\cite{MR2405237} and~\cite{nonhomeo}).
See also~\S\ref{subsec:lat_and_emb}.
\end{remark}
For $[\omega]\in \Omega_{\Klat}$, we 
put
$$
\NS^{[\per]}\;:=\;  \set{x\in \Klat}{\intnum{x}{\per}=0},
$$
which is a primitive sublattice of $\Klat$.
We then put
$$
\Omega_{\psi^\perp}:=\set{[\omega]\in \Omega_{\Klat}}%
{\intnum{\omega}{x}=0\;\; \textrm{for all} \;\; x\in  \lat}\;\;\subset\;\; \Omega_{\Klat}, 
$$
and  denote by $\Omega_{\psi^\perp}\sp\perspcond$ the set of all $[\omega]\in \Omega_{\psi^\perp}$
such that $\NS^{[\per]}$ satisfies the following conditions, which correspond to 
the properties (ii) and (iii) for $\lat$ in Proposition~\ref{prop:urabe}:
\begin{eqnarray}
&&\label{eqcond1}
\shortset{x\in  \NS^{[\per]}}{\intnum{x}{h}=0, \;x^2=-2}=\shortset{x\in \gen{\EEE}}{x^2=-2}
\quand\\
&&\label{eqcond2}
\shortset{x\in \NS^{[\per]}}{\intnum{x}{h}=1, \;x^2=0}=\emptyset.
\end{eqnarray}
Note that
the complement of $\Omega_{\psi^\perp}\sp\perspcond$ in $\Omega_{\psi^\perp}$ is 
a locally finite family of complex analytic subspaces.
From the surjectivity of $\tau_1$ and 
Propositions~\ref{prop:polarization} and~\ref{prop:geomroots}, we easily obtain the following:
\begin{proposition}
For any point $p\in \Omega_{\psi^\perp}\sp\perspcond$,
there exists a simple sextic $B$ with a marking $\phi: H^2(X_B,\Z)\isom \Klat$
such that   $\phi(\polB)=h$, $\phi(\EEE_B)=\EEE$, $\phi(\latB)=\lat$,
 and that the  period point of $(X_B, \phi)$ is $p$.
 
 Conversely,
 if $B$ is a simple sextic with a  marking $\phi: H^2(X_B,\Z)\isom \Klat$
 and  $\psi\sprime : \latB\isom \lat$ is an isomorphism of lattice data
 from   $\latdata(B)$ to  the lattice data  $[\EEE, h, \lat]$,
 then  the period point of $(X_B, \phi)$ is
contained in  $\Omega_{\psi^\perp}\sp\perspcond$,
where  $\psi:\lat\inj \Klat$ is  the primitive embedding
obtained from $\phi|\lat_B: \latB\inj \Klat$ via $\psi\sprime$.
\end{proposition}
We then  put
$$
\Omega_{\psi^\perp}\sp{\perspcond\perspcond}:=\set{[\omega]\in \Omega\sp\perspcond_{\psi^\perp}}{\NS^{[\per]}=\lat}.
$$
If $p\in \Omega_{\psi^\perp}\sp{\perspcond\perspcond}$,
then the corresponding simple sextic $B$ is lattice-generic.
It is obvious  that $\Omega_{\psi^\perp}\sp{\perspcond\perspcond}$ is  dense  in $\Omega_{\psi^\perp}\sp\perspcond$.
Hence we obtain the following:
\begin{corollary}\label{cor:perturb}
Given a simple sextic $B$,
we can obtain a lattice-generic simple sextic $B\sprime$
by an arbitrarily  small equisingular deformation  of $B$.
\end{corollary}
\section{Algorithms for a lattice type}\label{sec:algorithm}
Let $B$ be a  simple sextic.
\emph{Throughout this section,
we assume that $B$ is lattice-generic,
except for Corollary~\ref{cor:Yang}.}
In particular,
every splitting curve is pre-$Z$-splitting.
We present an algorithm to determine
the configuration type and  the sets $\ZZZ_1(B)$, $\ZZZ_2(B)$ and $\ZZZ_3(B)$
from  the lattice data $\latdata(B)=[\EEE_B, h_B, \latB]$  of $B$.
\par
\medskip
Recall that, 
for a splitting curve $\Gm$,
we denote by $\tlGmplus, \tlGmminus \subset X_B$ the lifts of $\Gm$.
For an irreducible component $B_i$ of $B$,
we denote by $\tilde{B}_i\subset X_B$ the reduced part of the strict transform of $B_i$, that is,
we put $\tilde{B}_i:=\tilde{B}_i^+=\tilde{B}_i^-$.
\par
\medskip
We denote by 
$j_B: W_B\to\Pt$  the \emph{Jung-Horikawa embedded resolution} (\emph{canonical embedded resolution}) of $B\subset \Pt$,
which  is the minimal succession of blowing ups 
such that the strict transform of $B$ is smooth and that
any  distinct  irreducible components of the total  transform 
 of $B$  with \emph{odd} multiplicities  do not intersect.
 (See~\cite[\S7 of Chap. III]{MR2030225}.)
 Then we have the \emph{finite} double covering $\tilde{\pi}_B: X_B\to W_B$
 that makes the following diagram commutative:
 $$
 \begin{array}{ccc}
 X_B & \maprightsp{\rho_B} & Y_B \\
 \mapdownleft{\hskip -10pt\tilde{\pi}_B} &&  \mapdownright{{\pi}_B} \\
 W_B &\maprightsp{j_B} & \Pt.
 \end{array}
 $$
For $P\in \Sing B$,
let $\EEE_P=\{e_1, \dots, e_r\}$ be the set of exceptional $(-2)$-curves on $X_B$ over $P$, 
which are indexed as in  Figure~\ref{figure:dynkin}.
For simplicity, we use the same letter for an exceptional $(-2)$-curve and its class, 
and consider $\EEE_P$ as a subset of $\blatB$.
Then $e_1, \dots, e_r$  form the  fundamental system of roots in  the sublattice $\gen{\EEE_P}$
of $\blatB$ associated with a K\"ahler class of $X_B$.
We denote by $e_1\dual, \dots, e_r\dual$ the dual basis of 
the dual lattice $\gen{\EEE_P}\dual\subset \gen{\EEE_P} \tensor \Q$. 
We have an orthogonal direct-sum decomposition
$$
\blatB=\gen{\polB}\oplus \bigoplus_{P\in \Sing B} \gen{\EEE_P}.
$$
Recall that $\latB$ is the primitive closure of $\blat_B$ in $H^2(X_B, \Z)$.
We consider the decomposition
\begin{equation}\label{eq:decomp}
\latB\otimes\Q\;=\;\blatB\otimes\Q\;=\;\gen{\polB}\otimes\Q \oplus  \bigoplus\gen{\EEE_P}\otimes\Q. 
\end{equation}
For $x\in \latB$,
we denote by $x_h\in \gen{\polB}\otimes\Q$ and $x_P\in  \gen{\EEE_P}\otimes\Q$
the components of $x$ under  the  direct-sum decomposition~\eqref{eq:decomp}.
The following is obvious:
\begin{lemma}\label{lem:hdeg0}
Let $D$ be an effective divisor on $X_B$ such that $\intnum{D}{\polB}=0$.
Then we have   $[D]\in \gen{\EEE_B}^+$.
In particular,  we have  $[D]\in \Sigma_B$ and $[D]_P\in \gen{\EEE_P}^+$ for any $P\in \Sing B$.
\end{lemma}
\begin{definition}\label{def:vsmooth}
We say that a vector $x\in \latB$ is \emph{$v$-smooth at $P\in \Sing B$}
if $x_P=0$ or $x_P=e_i\dual$ for some $i$.
We say that $x$ is \emph{$v$-smooth} if $x$ is $v$-smooth at every $P\in \Sing B$.
(The ``$v$" in \emph{$v$-smooth} stands for ``vector".)
\end{definition}
\begin{definition}\label{def:multiplicities}
Let $m_P(e_i\dual)$ denote the multiplicity of the curve $\tilde{\pi}_B (e_i)\subset W_B$
in the total transform of $B$ in $W_B$.
We also put $m_P(0):=0$.
Thus we have $m_P(x_P)$ for a vector $x\in \latB$ that is $v$-smooth at $P$.
\end{definition}
\begin{lemma}\label{lem:vsmooth}
Let $\tlGm$ be a lift of  a splitting curve $\Gm$, and 
let $P$ be a point of $\Sing B$.
Suppose that $P\notin \Gm$ or $\Gm$ is smooth at $P$.
Then the vector $[\tlGm]\in \latB$ is $v$-smooth at $P$ and $m_P([\tlGm]_P)$ is even.
\end{lemma}
This lemma is proved together with the following:
\begin{lemma}\label{lem:tGamma}
Let $\Gm\subset\Pt$ be a smooth splitting curve not contained in $B$.
Let $\Gm^W\subset W_B$ and $B^W \subset W_B$ be the strict transforms of $\Gm$ and $B$, 
respectively,  by  $j_B: W_B\to \Pt$,
and let $\tlB\subset X_B$ be the strict transform of $B$ by $\tilde{\rho}_B: \XB\to\Pt$.
Then we have
$$
\intnum{\tlGm^+}{\tlGm^-}_X=\intnum{\tlGm^+}{\tlB}_X=\intnum{\tlGm^-}{\tlB}_X=\intnum{\Gm^W}{B^W}_W/2, 
$$
where $\intnum{\phantom{i}}{\phantom{i}}_X$ and $\intnum{\phantom{i}}{\phantom{i}}_W$
denote the intersection numbers on $X_B$ and on $W_B$, respectively.
\end{lemma}
\begin{proof}[Proof of Lemmas~\ref{lem:vsmooth} and~\ref{lem:tGamma}]
The statement of Lemma~\ref{lem:vsmooth} is obviously true in the case where $P\notin \Gm$.
The proof of Lemma~\ref{lem:vsmooth} for the case where $\Gm$ is an irreducible component of $B$
is given in  Remark~\ref{rem:Bsmooth} below.

Suppose that $\Gm$ is splitting, is not contained in $B$,
and passes through  $P$.
Let  $F_1, \dots, F_m\subset W_B$ be the exceptional curves over $P$
of   $j_B$,
and let $m_k$ be the multiplicity of $F_k$ in the total transform of $B$ by $j_B$.
We denote by $T\subset W_B$ a sufficiently small tubular neighborhood of $j_B\inv (P)$,
and put  $\tilde{T}:=\tilde{\pi}_B\inv (T) \subset X_B$.
If $\intnum{\sum F_j}{ \Gm^W}_W>1$, then
the image $\Gm$ of $\Gm^W$ by $j_B$  would be singular at $P$.
Hence there exists  a unique irreducible component $F_i$ such that 
$\intnum{F_i }{ \Gm^W}=1$ and $\intnum{F_j }{ \Gm^W}=0$ for $j\ne i$.
Let $Q$ be the intersection point of $F_i$ and $\Gm^W$.
Note that  $\Gm^W$ is smooth at $Q$  and intersects $F_i$ transversely at $Q$.
Suppose that $Q\notin  B^W$, so that $\Gm^W$ is disjoint from $B^W$ in $T$.
Then, since $\Gm$ is splitting, the multiplicity $m_i$ 
is even and $\tilde{\pi}_B\inv (Q)$ consists 
of distinct two points.
Hence 
${\tlGm^+}$, ${\tlGm^-}$ and $\tlB$ are mutually disjoint in $\tilde{T}$, and 
Lemma~\ref{lem:vsmooth} holds by $m_P([\tlGm]_P)=m_i$.
Suppose that $Q\in  B^W$,
and let $n_Q$ be the intersection multiplicity of $B^W$ and $\Gm^W$ at $Q$.
Since $\Gm$ is splitting, $m_i+n_Q$ must be even.
Since $B^W\cap F_i\ne\emptyset$, $m_i$ is even.
Therefore $n_Q>1$, and hence
$B^W$ intersects $F_i$ transversely at $Q$;  in other words, 
$P$ is not of type $A_l$ with $l$ even.
Thus the pull-back of $F_i$ by  $\tilde{\pi}_B$ is irreducible,
and  Lemma~\ref{lem:vsmooth} holds by $m_P([\tlGm]_P)=m_i$.
In this case,
the intersection multiplicity  of 
${\tlGm^+}$ and ${\tlGm^-}$, or of $\tlGm^+$ and $\tlB$, or of $\tlGm^-$ and $\tlB$, 
at the point of $X_B$ over $Q$
is  equal to $n_Q/2$.
\end{proof}
\begin{remark}\label{rem:evensmooth}
If $P$ is of type $A_{l}$, then  the multiplicity $m_P(e_i\dual)$ is even for any $i$.
If $P$ is of other type, 
then $m_P(e_i\dual)$ is even if and only if  $e_i$ is subject to  the following restrictions:
\begin{itemize}
\item[] If $P$ is of type $D_{2k}$, then $i$ is even or $1$ or $2$.
\item[] If $P$ is of type $D_{2k+1}$, then  $i$ is odd or $1$ or $2$.
\item[]  If $P$ is of type $E_{6}$, then  $i\ne 1$.
\item[]  If $P$ is of type $E_{7}$, then  $i\ne 2,4,6$.
\item[]  If $P$ is of type $E_{8}$, then  $i\ne 2,4,6,8$.
\end{itemize}
\end{remark}
\begin{remark}\label{rem:Bsmooth}
Let $B_i$ be   an irreducible component of $B$ that contains $P\in \Sing B$ and is smooth at $P$.
Then the component $[\tlBi]_P \in \gen{\EEE_P}\tensor \Q$ is given as follows:
\begin{itemize}
\item[] If $P$ is of type $A_{2k-1}$,  then $[\tlBi]_P=e_k\dual$.
\item[] If $P$ is of type $D_{2k}$, then $[\tlBi]_P=e_1\dual$ or $[\tlBi]_P=e_2\dual$ or $[\tlBi]_P=e_{2k}\dual$.
\item[] If $P$ is of type $D_{2k+1}$, then  $[\tlBi]_P=e_{2k+1}\dual$.
\item[]  If $P$ is of type $E_{7}$, then  $[\tlBi]_P=e_{7}\dual$.
\end{itemize}
If $P$ is of another type,  every local irreducible components of $B$ at $P$ is singular.
\end{remark}
By  Remark~\ref{rem:Bsmooth}, we obtain the following:
\begin{lemma}\label{lem:Bsmooth}
Let $B_i$ be   an irreducible component of $B$ that contains $P\in \Sing B$ and is smooth at $P$.
Then  $[\tlBi]_P \in \gen{\EEE_P}\tensor \Q$ is not contained in $\gen{\EEE_P}$.
\end{lemma}
The following lemma is elementary,
but plays a crucial role in the following:
\begin{lemma}\label{lem:casebycase}
\setrmkakko
\rmkakko\label{lemrmkakko:ei}
For every $e_i\dual$, we have $(e_i\dual)^2<0$ and $e_i\dual\notin \gen{\EEE_P}^+$.

\rmkakko\label{lemrmkakko:eiej}
Suppose that
$e_i\dual-e_j\dual\in \gen{\EEE_P}^+$.
Then $(e_i\dual)^2 > (e_j\dual)^2$ or $e_i\dual=e_j\dual$.

\rmkakko\label{lemrmkakko:-9/2}
If   $e_i\dual$ is contained in $\gen{\EEE_P}$ and $m_P(e_i\dual)$ is even, then 
 $\intnum{{\iota_B}(e_i\dual)}{ e_i\dual}<-9/2$ holds.
\end{lemma}
\begin{proof} 
We have to prove this lemma only for the negative-definite root lattices
of type $A_l\; (l=1, \dots, 19)$,
$D_m\; (m=4, \dots, 19)$ and $E_n\; (n=6,7,8)$.
Hence  the assertions  can be proved by the case-by-case calculations.
For the proof, we use 
Remark~\ref{rem:evensmooth} above.
The involution $\iota_B$ is calculated by Remark~\ref{rem:iota} below.
(The author does not know any conceptual proof of this lemma.)
\end{proof}
\begin{remark}\label{rem:iota}
The involution $\iota_B$ on $\latB$ is determined by the $ADE$-type of $R_B$.
We have ${\iota_B}(\polB)=\polB$.
The action of  ${\iota_B}$   on   $\EEE_P$  is 
described as follows.
\begin{itemize}
\item[]  If $P$ is of type $A_l$, then
${\iota_B}(e_i)=e_{l+1-i}$.
\item[]   If  $P$ is of type $D_{2k}$, then  ${\iota_B}$ acts on $\EEE_P$ identically.
\item[]   If  $P$ is of type $D_{2k+1}$, then ${\iota_B}$ 
interchanges $e_1$ and $e_2$ and fixes  $e_3, \dots, e_{2k+1}$.
\item[] If  $P$ is of type $E_6$, then
${\iota_B}(e_1)=e_1$ and  ${\iota_B}(e_i)=e_{8-i}$ for $i=2, \dots, 6$.
\item[]  If  $P$ is of type $E_7$ or $E_8$, then ${\iota_B}$ acts on $\EEE_P$ identically.
\end{itemize}
\end{remark}
\begin{corollary}\label{cor:x=y}
Let $x\in \latB$ and $y\in\latB$ be  $v$-smooth vectors. 
If  $\intnum{x}{\polB}=\intnum{y}{\polB}$ and  $x^2=y^2$ hold and  $x-y$ is effective,
then $x=y$.
\end{corollary}
\begin{proof}
Since $x-y$ is effective and $\intnum{x-y}{\polB}=0$,
we have $x_P-y_P\in \gen{\EEE_P}^+$ for every $P\in\Sing B$
by Lemma~\ref{lem:hdeg0}.
Suppose that  $x\ne y$, and 
let $P\in \Sing B$ be a point 
such that  $x_P\ne y_P$.
Since $x$ and $y$ are  $v$-smooth,
each of $x_P$ and $y_P$ is $0$ or $e_i\dual$ for some $i$.
If $y_P=0$, then $x_P\ne 0$ and  $x_P\in \gen{\EEE_P}^+$,
which contradicts Lemma~\ref{lem:casebycase}~(1).
If $y_P\ne 0$, 
then we have $x_P^2>y_P^2$
by Lemma~\ref{lem:casebycase}~(1)~and~(2),
which  contradicts $x^2=y^2$.
\end{proof}
\begin{proposition}\label{prop:hdeg1}
Let $x\in \latB$ be a  $v$-smooth vector with $\intnum{x}{\polB}=1$ and $x^2=-2$.
Then $x$ is the class of a $(-2)$-curve that is mapped isomorphically to a line on $\Pt$.
\end{proposition}
\begin{proof}
By Riemann-Roch theorem for $X_B$,
we have an effective divisor $D$ on $X_B$ such that $x=[D]$.
Since $\intnum{x}{\polB}=1$, there exists a unique irreducible component $C$ of $D$ such that
$\intnum{C}{\polB}=1$.
Note that $C$ is mapped isomorphically to a line on $\Pt$,
and hence the image of $C$ is a splitting line.
Therefore $[C]^2=-2$ and  $[C]$ is  $v$-smooth by Lemma~\ref{lem:vsmooth}.
By Corollary~\ref{cor:x=y},
we have  $x=[C]$.
\end{proof}
We put 
\begin{eqnarray*}
\LLL_B&:=&\set{x\in \latB}{\textrm{$x$ is  $v$-smooth},\; \intnum{x}{\polB}=1, \;x^2=-2},\\
\LLL_B^b&:=&\set{x\in \LLL_B}{{\iota_B}(x)=x}, \quand \\
\LLL_B^l&:=&\set{x\in \LLL_B}{{\iota_B}(x)\ne x}.
\end{eqnarray*}

\begin{corollary}\label{cor:LLLb}
The map $B_i\mapsto [\tilde{B}_i]$ induces a bijection from the set of irreducible components $B_i$
of $B$ of degree $1$ to the set $\LLL_B^b$.
\end{corollary}
\begin{corollary}\label{cor:LLLl}
The set $\LLL_B^l$ is equal to the set  $\ZZZ_1(B)$ of the classes of lifts of $Z$-splitting lines.
\end{corollary}
Next we proceed to the study of $Z$-splitting conics.
\begin{proposition}\label{prop:hdeg2notblat}
Let $\tilde{C}\subset X_B$ be a curve that is mapped isomorphically to a smooth conic $C$ on $\Pt$.
Then $[\tilde{C}]\notin \blatB$.
\end{proposition}
\begin{proof}
We put $x:=[\tilde{C}]$.
Suppose that $C$ is an irreducible component of $B$.
Then $x_P\ne 0$ for some $P\in \Sing B$,
and hence $x\notin\blatB$ by Lemma~\ref{lem:Bsmooth}.
Suppose that $C$ is not contained in $B$.
Then 
$\intnum{{\iota_B}(x)}{ x}\ge 0$.
Since $C$ is smooth,
$x_P$ is $v$-smooth with $m_P(x_P)$ being even for every $P\in \Sing B$.
Since  $\intnum{x}{ \polB}=2$, 
we have $\intnum{{\iota_B}(x_h)}{x_h}=x_h^2=2$ and hence 
\begin{equation}\label{eq:ineqxiotax}
\intnum{{\iota_B}(x)}{ x}=2+\tsum_P \intnum{{\iota_B}(x_P)}{ x_P}\ge 0.
\end{equation}
Suppose that $x\in \blatB$
and hence $x_P\in \gen{\EEE_P}$ for any $P\in \Sing B$.
For any  $P\in C\cap \Sing B$,
we have  $x_P\ne 0$ and hence $\intnum{{\iota_B}(x_P)}{ x_P}<-9/2$ by Lemma~\ref{lem:casebycase}~(3).
By~\eqref{eq:ineqxiotax}, we therefore have $C\cap \Sing B=\emptyset$
and hence $\intnum{{\iota_B}(x)}{ x}=2$.
However, we  have  $\intnum{{\iota_B}(x)}{x}=6$ because  $\intnum{B}{C}=12$ on $\Pt$.
Thus we get a contradiction.
\end{proof}
\begin{proposition}\label{prop:hdeg2}
Let $x\in \latB$ be a  $v$-smooth vector such that $\intnum{x}{\polB}=2$, $x^2=-2$ and $x\notin\blatB$.
Then  one and only one of the following holds:
\begin{itemize}
\item[(i)] There exist $l_1, l_2\in \LLL_B$ such that
$x-(l_1+l_2)\in \gen{\EEE_B}^+$, or
\item[(ii)] $x$ is the class of a $(-2)$-curve $\tilde{C}$
that is a lift of a splitting conic $C$ on $\Pt$.
\end{itemize}
\end{proposition}
\begin{proof}
%
Note that $x$ is the class of an effective divisor of $X_B$.
We denote by $|D|$  the complete linear system  of effective divisors $D$ such that $x=[D]$.
The irreducible decomposition of each $D\in |D|$ is either
\begin{eqnarray}
&&D=\tilde{C}_1+\tilde{C}_2+\tsum e_i\quad
\textrm{with $\intnum{\tilde{C}_1}{\polB}=\intnum{\tilde{C}_2}{\polB}=1$ and $e_i\in \EEE_B$,}\quad\textrm{or} \label{eq:C1C2}\\
&&D=\tilde{C}+\tsum e_i\quad
\textrm{with $\intnum{\tilde{C}}{\polB}=2$ and $e_i\in \EEE_B$}. \label{eq:C}
\end{eqnarray}

Suppose that there exists $D\in |D|$ for which~\eqref{eq:C1C2} holds.
Since $B$ is assumed to be lattice-generic, 
we have $[\tilde{C}_1], [\tilde{C}_2]\in \latB$.
Since  $\tilde{C}_1$ and $\tilde{C}_2$ are mapped isomorphically to  lines on $\Pt$,
the vectors  $[\tilde{C}_1]$ and $[\tilde{C}_2]$ are  $v$-smooth with the square-norm $-2$.
Therefore $[\tilde{C}_1]$ and $[\tilde{C}_2]$ are in $\LLL_B$ and
thus the case  (i) occurs.

Suppose that there exists $D\in |D|$ for which~\eqref{eq:C} holds.
The image of $\tilde{C}$ in $\Pt$ is either a line or a smooth conic.
If the image were a line,
then $\tilde{C}$ would be  a strict transform of the line and hence $[\tilde{C}]$
would be contained in $\blatB$,
which contradicts the assumption.
Therefore $\tilde{C}$ is a lift of a splitting  conic $C$.
In particular,  $[\tilde{C}]\in \latB$ is a  $v$-smooth vector with $[\tilde{C}]^2=-2$.
By Corollary~\ref{cor:x=y}, we have $x=[\tilde{C}]$.
Therefore the case  (ii) occurs.

Suppose that both of the cases (i) and (ii) occur.
Then there exists $D_1\in |D|$ for which~\eqref{eq:C1C2} holds
and there exists $D_2\in |D|$ for which~\eqref{eq:C} holds.
By the argument above,
the existence of $D_2$ implies that $x$ is the class of a lift $\tilde{C}$  of a splitting  conic $C$,
and in particular $|D|$ consists of a single member $\tilde{C}$,
which contradicts the  existence of $D_1$.
Hence only one of (i) or (ii) occurs.
\end{proof}
We put
\begin{eqnarray*}
\CCC_B\sprime&:=&\set{x\in \latB}{\textrm{$x$ is  $v$-smooth},\; \intnum{x}{ \polB}=2, 
\;x^2=-2,\; x\notin\blatB},\quand \\
\CCC_B&:=&\set{x\in \CCC_B\sprime\,}{\textrm{for any $l_1, l_2\in \LLL_B$, we have $x-(l_1+l_2)\notin \gen{\EEE_B}^+$}}, \\
\CCC_B^b&:=&\set{x\in \CCC_B}{{\iota_B}(x)=x}, \\
\CCC_B^l&:=&\set{x\in \CCC_B}{{\iota_B}(x)\ne x}.
\end{eqnarray*}
\begin{corollary}\label{cor:CCCb}
The map $B_i\mapsto [\tilde{B}_i]$ induces a bijection from the set of irreducible components $B_i$
of $B$ of degree $2$ to the set $\CCC_B^b$.
\end{corollary}
\begin{corollary}\label{cor:CCCl}
The set $\CCC_B^l$ is equal to the set  $\ZZZ_2(B)$ of the classes of lifts of $Z$-splitting conics.
\end{corollary}
Next we study $Z$-splitting cubic curves. We put
\begin{eqnarray*}
\GGG_B&:=&\set{g\in \latB}{\textrm{$g^2=0$, $\intnum{g}{\polB}=3$, and 
$\intnum{g}{v}\ge 0$ for any $v\in \EEE_B\cup \LLL_B$}}, \\
\GGG_B^b&:=&\set{g\in \GGG_B}{\iota_B(g)= g}, \\
\GGG_B^l&:=&\set{g\in \GGG_B}{\iota_B(g)\ne g}.
\end{eqnarray*}
\begin{lemma}\label{lem:gnef}
Every $g\in \GGG_B$ is the class $[\tlE]$ of a member of an elliptic pencil $|\tlE|$ on $X_B$.
\end{lemma}
\begin{proof}
We have an effective divisor $D$ such that $g=[D]$ and $\dim|D|>0$.
We decompose $|D|$ into the  movable part $|M|$ and the fixed part $\Xi$.
Since $\dim|M|>0$, we have $\intnum{M}{\polB}\ge 2$ and hence $\intnum{\Xi}{\polB}\le 1$.
Therefore every irreducible component $C$ of $\Xi$ is either an element of $\EEE_B$ or
mapped isomorphically to a line of $\Pt$.
In the latter case, we have $[C]\in \LLL_B$.
Hence  $\intnum{C}{g}\ge 0$ holds for any irreducible component $C$ of $\Xi$
by the definition of $\GGG_B$.
Therefore  $g$ is nef.
Then, by  Nikulin~\cite[Proposition 0.1]{MR1260944},  we have $\Xi=\emptyset$ and 
there exists an elliptic pencil $|\tlE|$ on $X_B$ such that
$|D|=m |\tlE|$ for some integer $m>0$.
From $\intnum{g}{\polB}=3$, 
we obviously have  $m=1$.
\end{proof}
By Proposition~\ref{prop:cubic}, we see that  
every  $g\in \ZZZ_3 (B)$ is nef and hence
satisfies $\intnum{g}{v}\ge 0$ for any $v\in \EEE_B\cup \LLL_B$.
Combining Proposition~\ref{prop:cubic} and Lemma~\ref{lem:gnef}, we obtain the following:
\begin{corollary}\label{cor:GGG}
We have $\GGG^l_B=\ZZZ_3 (B)$.
\end{corollary}
\begin{proposition}\label{prop:cubic2}
Suppose that $B$ does not have any irreducible components of degree $\le 2$.
Then $B$ is irreducible if and only if $\GGG_B^b=\emptyset$.
\end{proposition}
\begin{proof}
Suppose that $B$ is reducible.
Then $B$ is a union of two irreducible cubic curves $E_0$ and $E_\infty$.
Note that, for each $P\in E_0\cap E_\infty$,
either $E_0$ or $E_\infty$ is smooth at $P$.
Let $\PPP\subset |\OOO_{\Pt}(3)|$ be the pencil spanned by $E_0$ and $E_\infty$.
%
Examining the Jung-Horikawa resolution $j_B: W_B\to \Pt$ explicitly,
we see that $j_B$ resolves the  base points of $\PPP$, and hence 
we obtain  an elliptic fibration 
$$
\phi_{\PPP}: W_B\to  \P^1
$$
on $ W_B$ such that, by $j_B: W_B\to \Pt$, 
the general fiber of $\phi_{\PPP} $ is mapped to
a member of ${\PPP}$,  and $\phi_{\PPP} \inv (0)$ and 
$\phi_{\PPP} \inv (\infty)$ are  mapped to $E_0$ and $E_\infty$, respectively.
Moreover the branching locus of $\tilde{\pi}_B: X_B\to W_B$ is contained in 
$\phi_{\PPP} \inv (0)\cup \phi_{\PPP} \inv (\infty)$.
Indeed,
suppose that $E_0$ is smooth at $P\in E_0\cap E_\infty$,
and let $F_1, \dots, F_m$ be the exceptional curves of $j_B$ over $P$.
There exists a unique $F_i$ among them
that intersects the strict transform of $E_0$.
This component $F_i$ becomes a section of $\phi_{\PPP}$,
and  the other components are mapped to $\infty$ by $\phi_{\PPP}$.
The multiplicity of $F_i$ in the total transform of $B$ is even,
and hence $\tilde{\pi}_B$ does not ramify along the section $F_i$.

Thus we have an elliptic fibration ${\psi}_{\PPP}: X_B\to \P^1$ that fits in
 a commutative diagram
$$
\begin{array}{ccc}
X_B & \maprightsp{\tilde{\pi}_B} & W_B \\
\mapdownleft{\hskip -3.5mm\psi_{\PPP}} && \mapdownright{\phi_{\PPP} } \\
\P^1 & \maprightsp{\bar{\pi}_B} &  \P^1,
\end{array}
$$
where $\bar{\pi}_B: \P^1\to\P^1$ is the double covering branching at $0\in \P^1$ and $\infty\in \P^1$.
Let $\tlE\subset X_B$ be the general fiber of 
the elliptic fibration $\psi_{\PPP}: X_B\to\P^1$.
Since $\tlE$ is nef, we see that 
 $g:=[\tlE]\in \latB$ is an element of $\GGG_B^b$.
\par
\medskip
Conversely, suppose that $g\in \GGG_B^b$.
By Lemma~\ref{lem:gnef},
we have an elliptic fibration  $\psi: X_B\to \P^1$ 
such that the class of its general fiber $\tlE$ is $g$.
Since ${\iota_B}(g)=g$, the involution ${\iota_B}$  preserves this elliptic fibration.
Therefore $\psi: X_B\to \P^1$ is obtained from 
an elliptic fibration $\phi: W_B\to \P^1$ 
on $W_B=X_B/\gen{{\iota_B}}$ by the base change $\bar\pi: \P^1\to\P^1$
of degree $2$.
Since the branch points of $\bar\pi$ consists of two points,
the branch curve  of $\tilde{\pi}_B: W_B\to X_B$ is contained in the union of two fibers of 
$\phi: W_B\to \P^1$, each of which is mapped to a cubic irreducible component of $B$.
\end{proof}
\begin{remark}
Suppose that $g\in \GGG_B^b$.
Note that
a point $P\in \Sing B$ of type $A_1$ is an intersection point of 
the irreducible components 
$E_0$ and  $E_\infty$ of $B$
if and only if $g_P\ne 0$.
Therefore we can recover the configuration type of $B$ from $g$.
\end{remark}
\begin{remark}
There are additional necessary conditions for $\degs B$ to be $[3,3]$,
which are helpful in calculation.
If $\degs B=[3,3]$, then $R_B$ consists of the following $ADE$-types;
$A_2, A_{2k-1}, D_5, D_{2k}, E_7$, and 
moreover, for a point  $P\in \Sing B$  of type $t_P$, 
the component $g_P$ of  the vector  $g\in \GGG_B^b$ should satisfy  the following:
$$
\renewcommand{\arraystretch}{1.2}
\begin{array}{c|c|c|c|c|c|c|c|}
t_P &A_2 & A_1 & A_{2k-1}\;\; (k>1) &D_5 &D_4 &D_{2k}\;\; (k>2) & E_7 \\
g_P & 0 & 0\;\;\rmor\;\; e_1\dual &e_k\dual &e_5\dual &e_1\dual, e_2\dual \;\;\rmor\;\;  e_4\dual &e_1\dual\;\;\rmor\;\; e_2\dual &e_7\dual 
\end{array}
$$
\end{remark}
\par
\medskip
We now interpret these geometric results to lattice-theoretic results.
\begin{definition}
A fundamental system of roots is called \emph{irreducible}
if the corresponding Dynkin diagram is connected.
\end{definition}
Let $\latdata=[\EEE, h, \lat]$ be  lattice data.
We put
$$
\blat:=\gen{h}\oplus\gen{\EEE}.
$$
We denote by $\sing\latdata$ the set of irreducible components of $\EEE$,
and let
$$
\EEE=\bigsqcup_{P\in \sing \latdata} \EEE_P
$$
be the irreducible decomposition of $\EEE$.
We then have an orthogonal direct-sum decomposition
$$
\lat\tensor\Q\;=\;\gen{h}\tensor\Q \;\oplus\; \bigoplus\; \gen{\EEE_P}\tensor \Q.
$$
We say that $x\in \lat$ is \emph{${\bf v}$-smooth at $P\in \sing \latdata$}
if the component $x_P\in \gen{\EEE_P}\tensor \Q$ of $x$ 
is either $0$ or equal to some $e_i\dual\in \gen{\EEE_P}\dual$,
where $e_1\dual, \dots, e_r\dual$ are the basis of $\gen{\EEE_P}\dual$
dual to the basis $\EEE_P=\{e_1, \dots, e_r\}$ of $\gen{\EEE_P}$.
We say that $x$ is \emph{${\bf v}$-smooth} if $x\in \lat$ is ${\bf v}$-smooth at every $P\in \sing \latdata$.
%
\begin{remark}
The notion ``\,\emph{${\bf v}$-smooth\,}" is the lattice theoretic version of the geometric notion ``\emph{\,$v$-smooth\,}" defined 
in Definition~\ref{def:vsmooth}.
\end{remark}
We also define an involution $\iota$ of $\lat\tensor\Q$
by Remark~\ref{rem:iota} with $\latB$ replaced by $\lat$ and $\iota_B$ replaced by $\iota$.
Then we can define the subsets 
$$
\LLL^l(\latdata), \;\;
\LLL^b(\latdata), \;\;
\CCC^l(\latdata), \;\;
\CCC^b(\latdata), \;\;
\GGG^l(\latdata), \;\;
\GGG^b(\latdata)
$$
of $\lat$ in the same way as the sets $\LLL^l(B), \LLL^b(B), \CCC^l(B), \CCC^b(B), \GGG^l(B), \GGG^b(B)$
with $\latB$ replaced by $\lat$, $\polB$ replaced by $h$, $\blatB$ replaced by $\blat$, 
$v$-smooth replaced  by ${\bf v}$-smooth, and $\iota_B$ replaced by $\iota$.
If $\phi:\latB\isom\lat$ is an isomorphism of lattice data from $\latdata (B)$ to $\latdata$, then 
$\phi$ maps $\LLL^l(B)$, $\LLL^b(B)$, $\CCC^l(B)$, $\CCC^b(B)$, $\GGG^l(B)$, $\GGG^b(B)$
to 
$\LLL^l(\latdata)$, $\LLL^b(\latdata)$, $\CCC^l(\latdata)$, $\CCC^b(\latdata)$, $\GGG^l(\latdata)$, $\GGG^b(\latdata)$
bijectively, respectively.
In other words, 
these subsets 
of $\latB$
are determined only by the lattice data of $B$.
\par
\medskip
Thus we have shown that the configuration type 
of  a lattice-generic simple sextic $B$ is determined by the lattice type of $B$.
Hence we obtain the following,
which has been proved in Yang~\cite{MR1387816}.
\begin{corollary}\label{cor:Yang}
Let $B_1$ and $B_2$ be simple sextics (not necessarily lattice-generic) 
in the same lattice type.
Then $B_1\relconfig B_2$ holds.
\end{corollary}
\begin{proof}
There exist lattice-generic simple sextics $B_1\sprime$ and $B_2\sprime$ such that
$B_1\sprime\releqs B_1$ and $B_2\sprime\releqs B_2$.
Since $B_1\sprime\rellat B_2\sprime$, we have $B_1\sprime\relconfig B_2\sprime$ by the above arguments.
Thus $B_1\relconfig B_2$ follows.
\end{proof}
We have also shown that the subsets $\ZZZ_1(B)$, $\ZZZ_2(B)$ and $\ZZZ_3(B)$ of $\latB$
for  a lattice-generic simple sextic $B$ are determined only by the lattice type of $B$,
and hence Theorem~\ref{thm:ZZZ} is proved.
\begin{computation}\label{comp:latdata}
We have already obtained the complete list of lattice data of simple sextics by Computation~\ref{comp:yang}.
For  each piece $\latdata=[\EEE, h, \lat]$ of the lattice data  in this list, 
we make the following calculation.

We compute the subsets $\LLL^l(\latdata), \LLL^b(\latdata), \CCC^l(\latdata), \CCC^b(\latdata)$ of $\lat$.
If $\LLL^b(\latdata)=\CCC^b(\latdata)=\emptyset$, then 
we calculate $\GGG^b(\latdata)$.
Thus we determine the configuration type
containing the lattice type of the lattice data $\latdata$.
We then calculate
$$
\bBlat:=\begin{cases}
\blat+\gen{\LLL^b(\latdata)}+\gen{\CCC^b(\latdata)} & \textrm{if $\LLL^b(\latdata)\ne \emptyset$ or $\CCC^b(\latdata)\ne\emptyset$, } \\
\blat+\gen{\GGG^b(\latdata)} & \textrm{if $\LLL^b(\latdata)=\CCC^b(\latdata)=\emptyset$, } \\
\end{cases}
$$
and $F_\latdata:=\lat/\bBlat$.

Suppose that $\LLL^l(\latdata)\ne \emptyset$ or $\CCC^l(\latdata)\ne\emptyset$.
We confirm that $F_{\latdata}\ne 0$ and that the equality
$\lat=\bBlat +\gen{\LLL^l(\latdata)}+\gen{\CCC^l(\latdata)}$
holds.
Suppose that $\LLL^l(\latdata)=\CCC^(\latdata)=\emptyset$ but $F_{\latdata}\ne 0$.
We then calculate  $\GGG^l(\latdata)$, and confirm that $\GGG^l(\latdata)$ consists of two elements,  that 
$\lat=\bBlat +\gen{\GGG^l(\latdata)}$
holds, and that $\lat/\blat$ is cyclic of order $4$.
\end{computation}
\begin{remark}
In order to determine whether or not two lattice types are contained in the same configuration type,
we have to use the \emph{combinatorial} definition of the configuration type,
which is given in~\cite[Remark 3]{MR2409555} for example.
\end{remark}
By this calculation,
we prove Theorems~\ref{thm:LZkplets},~\ref{thm:FB} and the first part of Theorem~\ref{thm:sp}.
We also obtain  the complete list of lattice data
of $Z$-splitting pairs $(B, \Gm)$ with $\deg\Gm\le 2$,
or  with $z_1(\lattype(B))=z_2(\lattype(B))=0$, $F_B\ne 0$
and $\Gm$ being smooth cubic.
Our next task is to determine the relation of specializations 
among the lattice data
of $Z$-splitting pairs.
\section{Specialization of lattice types}\label{sec:specialization}
For the study of specialization of lattice types, we need to refine the  period map $\tau_1: \MMM_1\to \Omega_{\Klat}$.
%
(See~\cite[Chap.~VIII]{MR2030225} or~\cite{MR785231}.)
Consider the real vector bundle
$R^2 \pi_{1*}\R$ of rank $22$ over the non-Hausdorff moduli space $\MMM_1$,
where $\pi_1: \XXX_1\to \MMM_1$ is the universal family of (marked) $K3$ surfaces.
A point of this vector bundle is given by
$(t, x)$, where $t\in \MMM_1$ and $x\in H^2(X_t, \R)$.
We then put
$$
\MMM_2 :=\set{(t, \kappa)\in R^2 \pi_{1*}\R }{ \textrm{$\kappa$ is a K\"ahler class of $X_t$}},
$$
that is,  $\MMM_2$
is the base space of the universal family of the triples
$(X, \phi, \kappa)$,
where $(X, \phi)$ is a marked  $K3$ surface and 
$\kappa$
 is a  K\"ahler class  of $X$.
\par
For a point  $[\per]$ of $\Omega_{\Klat}$,
we put
\begin{eqnarray*}
H^{[\per]} &:=& \set{x\in \Klat\tensor\R}{\intnum{x}{\per}=0},\\
\NS^{[\per]} &:=& H^{[\per]}\cap \Klat \quad\textrm{(as defined in the previous section)},\\
D^{[\per]}&:=& \set{d\in \NS^{[\per]}}{d^2=-2}, \\
\Gamma^{[\per]} &:=& \set{x\in H^{[\per]}}{ x^2>0}, \\
\zeroGamma^{[\per]} &:=& \set{x\in \Gamma^{[\per]} }{ \intnum{x}{d}\ne 0\;\;\textrm{for all}\;\; d\in D^{[\per]}}.
\end{eqnarray*}
We then put
\begin{eqnarray*}
H\persp_{\Klat}&:=&\set{([\per], x)\in \persp_{\Klat}\times ({\Klat}\tR)}{x\in H^{[\per]}},
\quad \\
K\persp_{\Klat}&:=&\set{([\per], x)\in \persp_{\Klat}\times ({\Klat}\tR)}{x\in \Gamma^{[\per]}},
\quand \\
 \spzero  K\persp_{\Klat}&:=&\set{([\per], x)\in \persp_{\Klat}\times ({\Klat}\tR)}{x\in \zeroGamma^{[\per]}}.
\end{eqnarray*}
We have a commutative diagram
\begin{equation}\label{eq:overpersp}
\begin{array}{ccccc}
  \spzero  K\persp_{\Klat} \hskip -.3cm & \inj  &  K\persp_{\Klat} & \inj & \hskip -.3cmH\persp_{\Klat} \\
  &{}_{\Pi_{\persp}}\searrow\;\;\;\phantom{\Pi_{\persp}} & \mapdown &\phantom{\Pi_{\persp}}\;\;\;\swarrow\phantom{\Pi_{\persp}} & \\
  &&\phantom{,,}\persp_{\Klat}\;, &&
\end{array}
\end{equation}
where the maps to $\persp_{\Klat}$ are the projection onto the first factor.
Note that $ K\persp_{\Klat}$ and $H\persp_{\Klat} $ are locally trivial fiber spaces over $\persp_{\Klat}$.
We have the following:
\begin{lemma}[Corollary 9.2 in Chapter VIII of \cite{MR2030225}]
The space $\spzero K\persp_{\Klat}$ is open in $K\persp_{\Klat}$,
and hence the projection $\Pi_{\Omega}$ is an open immersion.
\end{lemma}
Let  $t$ be a point of $\MMM_1$, 
and let 
$$
[\per_t]:=\tau_1(t)\;\;\in\;\; \persp_{\Klat}
$$
be the period point of  $(X_t, \phi_t)$.
Then 
the marking   $\phi_t: H^2 (X_t, \Z)\isom \Klat$
maps
$H_{X_t}$ to $H^{[\per_t]}$,
$\Gamma_{X_t}$ to $\Gamma^{[\per_t]}$,
$\NS (X_t)$ to $\NS^{[\per_t]}$,
$D_{X_t}$ to $D^{[\per_t]}$,  and hence 
$\phi_t$ maps
$\zeroGamma_{X_t}$ to $\zeroGamma^{[\per_t]}$. 
Since every K\"ahler class of $X_t$ is contained in 
the K\"ahler cone $\Kahler_{X_t}\subset\zeroGamma_{X_t}$,
we can define a map 
$$
\tau_2 : \MMM_2 \to \spzero K\persp_{\Klat},
$$
which is called the \emph{refined period map}, by 
$$
\tau_2 (t, \kappa):=(\tau_1 (t), \phi_t(\kappa)).
$$
Then we obtain a commutative diagram
\begin{equation}\label{diagram:Y}
\renewcommand{\arraystretch}{1.4}
\begin{array}{ccc}
\MMM_2 & \maprightsp{\tau_2} & \spzero K\persp_{\Klat}\\
\mapdownleft{\hskip -14pt\Pi_{\MMM}} &&\mapdownright{\Pi_{\Omega}} \\
\MMM_1 & \maprightsp{\tau_1} &\phantom{,}\persp_{\Klat}, 
\end{array}
\end{equation}
where the  vertical arrows
$\Pi_{\MMM}$ and $\Pi_{\Omega}$ 
are the forgetful maps.
\par
\medskip
The following  plays a crucial role in 
the study of specialization of lattice types:
\begin{theorem}[Theorems 12.3 and 14.1 in Chapter VIII of \cite{MR2030225}]\label{thm:refined}
The refined period map $\tau_2$ is an isomorphism.
\end{theorem}
The specialization of lattice types of simple sextics and $Z$-splitting pairs 
can be described by \emph{geometric embeddings} of lattice data.
\begin{definition}\label{def:geometric-specialization}
Let $\latdata=[\EEE, h, \lat]$ 
and $\latdata_0=[\EEE_0, h_0, \lat_0]$
be  lattice data.
By a \emph{geometric embedding} of $\latdata$ into $\latdata_0$,
we mean  
a primitive embedding    $\sigma : \lat\inj \lat_0$ of the lattice $\lat$ into the lattice $\lat_0$
that satisfies  $\sigma (h)=h_0$ and $\sigma (\EEE)\subset \gen{\EEE_0}^+$.
\end{definition}
\begin{definition}\label{def:geometric-specializationZP}
Let $\latdataZP=[\EEE, h, \lat, S]$
and $\latdataZP_0=[\EEE_0, h_0, \lat_0, S_0]$ 
be  extended  lattice data.
A \emph{geometric embedding} of $\latdataZP$ into $\latdataZP_0$
is a geometric embedding $\sigma : \lat\inj \lat_0$ 
of $[\EEE, h, \lat]$ into $[\EEE_0, h_0, \lat_0]$
such that we have
$$
\sigma (S) \subset S_0+ \gen{\EEE_0}^+:=(v_0^++\gen{\EEE_0}^+)\cup(v_0^-+\gen{\EEE_0}^+),
\quad\textrm{where $S_0=\{v_0\sp\pm\}$}.
$$
\end{definition}
Let $f:\BBB\to\Delta$ be an analytic family of
simple sextics,
where $f$ is the projection from $\BBB\subset \P^2\times \Delta$ to $\Delta$,
 and  $B_t:=f\inv (t)$ is a simple sextic on $\Pt\times\{t\}$ for any $t\in \Delta$.
Suppose that $f$ is equisingular over $\Delta\sptimes$.
We define a geometric embedding
$$
\sigma_{\BBB, t}: \lat_{B_t} \inj \lat_{B_0}
$$
of the lattice data $\latdata(B_t)$ with $t\ne 0$ into the lattice data $\latdata(B_0)$
as follows.
We consider the double cover
$$
\YYY_{\BBB}\to \P^2\times \Delta
$$
branching exactly along $\BBB$.
Note that every fiber of $\YYY_{\BBB}\to \Delta$ is birational to
a $K3$ surface.
Therefore, by Kulikov~\cite{MR0506295}, 
there exists a birational transformation
$\XXX_{\BBB}\to \YYY_{\BBB}$
such that the composite holomorphic map
$$
\pi_{\BBB}:\XXX_{\BBB}\to \Delta
$$
 is  a smooth family of $K3$ surfaces.
Note that the fiber of $\pi_{\BBB}$ over $t\in \Delta$ 
is isomorphic to $X_{B_t}$.
Note also that $\XXX_{\BBB}$ has a line bundle $\LLL_{\BBB}$ such that
the class of the restriction of $\LLL_{\BBB}$ to $X_{B_t}=\pi_{\BBB}\inv (t)$
is equal to  $h_{B_t}\in H^2(X_{B_t}, \Z)$ for any $t\in \Delta$.
Then we have a trivialization
$$
R^2 \pi_{\BBB*}\Z\cong \Delta\times \Klat,
$$
which induces markings $H^2(X_{B_t}, \Z)\isom \Klat$ for any $t\in \Delta$.
Using this trivialization,
we obtain a primitive  embedding $\sigma_{\BBB, t}: \lat_{B_t} \inj \lat_{B_0}$ of lattices
by the specialization homomorphism
$$
H^2(X_{B_t}, \Z)\isom H^2(X_{B_0}, \Z).
$$
This   $\sigma_{\BBB, t}$ induces a  geometric embedding 
of the lattice data $\latdata(B_t)$ for $t\ne 0$ into the lattice data $\latdata(B_0)$.
Indeed, $\sigma_{\BBB, t}$ maps $h_{B_t}$ to $h_{B_0}$
because the polarizations on $X_{B_t}$ form a family $\LLL_{\BBB}$. 
Moreover
any exceptional  $(-2)$-curve on $X_{B_t}$ $(t\ne 0)$ degenerates into an effective divisor
on $X_{B_0}$, whose reduced irreducible components must be  
exceptional  $(-2)$-curves on $X_{B_0}$ because its degree with respect to the polarization $h_{B_0}$
is zero.
\begin{proposition}\label{prop:ifspZP}
Let $\{(B_t, \Gm_t)\}_{t\in \Delta}$  be  an analytic family of $Z$-splitting pairs
that is equisingular over $\Delta\sptimes$.
Then the geometric  embedding $\sigma_{\BBB, t}$
of $\latdata (B_t)$ with $t\ne 0$ into $\latdata (B_0)$
yields
a geometric embedding  of the extended lattice data 
$\latdataZP(B_t, \Gm_t)$ with $t\ne 0$  into  the extended lattice data $\latdataZP (B_0, \Gm_0)$.
\end{proposition} 
\begin{proof}
Since $\Gm_t$ degenerates into $\Gm_0$,
the curve $\tlGm_{t}^{+}\subset X_{B_t}$  for $t\ne 0$ degenerates into
an effective divisor on $X_{B_0}$ that is the sum of  $\tlGm_0^{+}$ (or $\tlGm_0^{-}$) and some exceptional $(-2)$ curves
on $X_{B_0}$.
Hence the geometric embedding $\sigma_{\BBB, t}: \lat_{B_t} \inj \lat_{B_0}$ of 
$\latdata(B_t)$  into  $\latdata(B_0)$ constructed above 
satisfies $\sigma_{\BBB, t}([\tlGm_{t}^{+}])\in [\tlGm_0^{+}]+\gen{\EEE_{B_0}}^+$ or 
$\sigma_{\BBB, t}([\tlGm_{t}^{+}])\in [\tlGm_0^{-}]+\gen{\EEE_{B_0}}^+$.
\end{proof} 
\begin{corollary}\label{cor:ifspZP}
Let $\lattypeZP_0$ and $\lattypeZP$ be lattice types of $Z$-splitting pairs,
and let  $\latdataZP_0$ and $\latdataZP$ be the corresponding extended lattice data. 
If $\lattypeZP_0$ is a specialization of $\lattypeZP$, then 
there exists a geometric embedding  of $\latdataZP$ into $\latdataZP_0$.
\end{corollary} 
Since 
a geometric embedding $\sigma: \lat_B\inj \lat_{B_0}$ of $\latdataZP(B, \Gm)$ into $\latdataZP(B_0, \Gm_0)$
induces a homomorphism of finite abelian groups 
$G_B\to G_{B_0}$ that maps $([\tlGm^+]\bmod \blatB)\in G_B$ to $([\tlGm_0^+]\bmod \blat_{B_0})\in  G_{B_0}$ or 
$([\tlGm_0^-]\bmod \blat_{B_0}) \in  G_{B_0}$,
we obtain the following:
\begin{corollary}
If $\lattypeZP_0=\lattypeZP(B_0, \Gm_0)$ is a specialization of $\lattypeZP=\lattypeZP(B, \Gm)$,
then the class-order of $\lattypeZP_0$ is a divisor of  the class-order of $\lattypeZP$.
\end{corollary}
In order to show that
the existence of a geometric embedding of lattice data
with certain properties is \emph{sufficient}
for the existence of the specialization, we prepare two easy lemmas.
\par
\medskip
Let $\pi: \XXX\to \Delta$ be a smooth family of $K3$ surfaces.
We put $X_t:=\pi\inv (t)$.
\begin{lemma}\label{lem:LLL}
Let $s$ be a section of $R^2\pi_*\Z$.
If $s_t:=s|X_t\in H^2(X_t, \Z)$ is contained in $H^{1,1}(X_t)$ for any $t\in \Delta$, then
there exists a line bundle $\LLL_{\XXX}$ on $\XXX$ such that 
the class of the restriction $\LLL_t:=\LLL_{\XXX}|X_t$ is equal to $s_t$.
\end{lemma}
\begin{proof}
This follows immediately from the commutative diagram
\begin{equation*}\label{eq:diagmamPic}
\begin{array}{ccccc}
H^1(\XXX, \OOO\sptimes) 
&\;\;\to\;\; &
H^2(\XXX, \Z)
&\;\;\to\;\; &
H^2(\XXX, \OOO) \\
&&\mapdownright{\wr}
&&\mapdownright{\wr}\\
&&H^0(\Delta, R^2\pi_* \Z)
&\;\;\to\;\;&H^0(\Delta, R^2\pi_* \OOO),
\end{array}
\end{equation*}
where the horizontal sequences are induced from  the exponential exact sequence
$0\to\Z\to\OOO\to\OOO\sptimes\to 0$.
\end{proof}
\begin{lemma}\label{lem:V}
Let $\LLL_{\XXX}$ be a line bundle on $\XXX$,
and we put $\LLL_t:=\LLL_{\XXX}|X_t$ for $t\in \Delta$.
If $h^1(X_0, \LLL_0)=0$ and $h^0(X_0, \LLL_0)>0$,
then there exists a linear subspace $V\subset H^0(\XXX, \LLL_{\XXX})$
of dimension equal to $h^0(X_0, \LLL_0)$
such that, after replacing $\Delta$ with a smaller disc if necessary,
the restriction homomorphism $H^0(\XXX, \LLL_{\XXX})\to H^0(X_t, \LLL_t)$
maps $V$ isomorphically onto $ H^0(X_t, \LLL_t)$ for any $t\in \Delta$.
\end{lemma}
\begin{proof}
From $h^0(X_0, \LLL_0)>0$,
we have $h^2(X_0, \LLL_0)=0$.
By the semi-continuity theorem, 
the assumption $h^1(X_0, \LLL_0)=0$ implies $h^1(X_t, \LLL_t)=0$ 
and $h^0(X_t, \LLL_t)=h^0(X_0, \LLL_0)$
for $t$ in a sufficiently small neighborhood of $0$,
because $\LLL_t^2\in \Z$ is constant.
Hence, by replacing $\Delta$ with a smaller disc if necessary,
we can assume that $H^1(\XXX, \LLL_{\XXX})=0$ and hence
$H^1(\XXX, \LLL_{\XXX}(-X_t))=0$ holds for any $t\in \Delta$,
because $ \LLL_{\XXX}\cong  \LLL_{\XXX}(-X_t)$ on $\XXX$.
Therefore the restriction homomorphism $H^0(\XXX, \LLL_{\XXX})\to H^0(X_t, \LLL_t)$
is surjective for any $t\in \Delta$.
\end{proof}
%
The following proposition seems to be well-known.
We present, however,  a complete proof,
because it illustrates how the refined period map is used for the study of specializations of simple sextics,
and it sets up various tools 
necessary for the proof of  Proposition~\ref{prop:onlyifspZP} below.
\begin{proposition}\label{prop:onlyifsp}
Let $\latdata_0=[\EEE_0, h_0, \lat_0]$ and $\latdata=[\EEE, h, \lat]$ be 
lattice data of simple sextics.
Suppose that
a simple  sextic $B_0$ with an isomorphism
$\alpha_0: \lat_{0}\isom\lat_{B_0}$ of 
lattice data from $\latdata_0$ to $\latdata(B_0)$ is given.
If  a geometric embedding $\sigma: \lat\inj \lat_0$ of $\latdata$ into $\latdata_0$ is given,
then  we can construct an analytic family $f:\BBB\to\Delta$ of simple sextics $B_t=f\inv (t)$
and isomorphisms 
$$
\alpha_t: \lat\isom\lat_{B_t}
$$
of  lattice data from $\latdata$ to   $\latdata(B_t)$ for $t\ne 0$ 
that  satisfy the following:
\begin{itemize}
\item[(i)] the central fiber $f\inv (0)$ of $f$ is the given simple sextic $B_0$,
\item[(ii)] $f$ is equisingular over $\Delta\sp\times$, 
\item[(iii)] for $t\ne 0$, the composite $\alpha_0\inv \circ \sigma_{\BBB, t}\circ\alpha_t: \lat\inj \lat_0$
is equal to the given geometric embedding $\sigma$ of $\latdata$ into $\latdata_0$, and
\item[(iv)] the locus of all $t\in \Delta$ such that $B_t$ is lattice-generic is dense in $\Delta$.
\end{itemize}
\end{proposition}
\begin{proof}
For simplicity,
we put
$$
X_0:=X_{B_0}.
$$
We fix a marking
$$
\phi_0 : H^2(X_0, \Z)\isom \Klat.
$$
By $\alpha_0$ and $\phi_0$, we 
obtain  a primitive embedding  
$$
\psi: \lat_0\inj \Klat.
$$
By the composition of $\sigma$ and $\psi$, we obtain a primitive embedding  
$$
\psi\circ\sigma: \lat\inj \Klat.
$$
From now on, we consider $\lat$ and $\lat_0$ as primitive sublattices of $\Klat$ by
$\psi\circ\sigma$ and $\psi$, respectively:
$$
 \lat\;\;\subset\;\;\lat_0\;\;\subset\;\; \Klat.
$$
In particular,  we have $h=h_0=\phi_0(h_{B_0})\in \Klat$ and
$\EEE_0=\phi_0(\EEE_{B_0})\subset \Klat$,  $\EEE\subset\gen{\EEE_0}^+\subset  \Klat$.
Moreover  we have  inclusions of complex submanifolds
$$
\Omega_{\psi\sp\perp}\subset \Omega_{(\psi\circ \sigma)\sp\perp} \subset \Omega_{\Klat}.
$$
Let  $[\eta_0]\in \Omega_{\Klat}$ be the period point of the marked $K3$ surface $(X_0, \phi_0)$.
Then  $[\eta_0]$ is a point of $\Omega_{\psi\sp\perp}$.
We choose an analytic embedding
$$
\delta: \Delta\inj \Omega_{(\psi\circ \sigma)\sp\perp} 
$$
of the open unit disk $\Delta\subset \C$ into  a sufficiently small neighborhood of $[\eta_0]$
such that $\delta(0)=[\eta_0]$, and that 
$\delta\inv (\Omega_{(\psi\circ \sigma)\sp\perp}\sp\perspcond)=\Delta\setminus\{0\}$ holds, and
that $\delta\inv (\Omega_{(\psi\circ \sigma)\sp\perp}\sp{\perspcond\perspcond})$ is dense in $\Delta$.
 (These properties  can be  achieved because $\Omega_{(\psi\circ \sigma)\sp\perp}\sp\perspcond$ is open in
$\Omega_{(\psi\circ \sigma)\sp\perp}$ and 
$\Omega_{(\psi\circ \sigma)\sp\perp}\sp{\perspcond\perspcond}$ is dense in $\Omega_{(\psi\circ \sigma)\sp\perp}$.)
We write
$$
\delta(t)=[\eta_t]\in \persp_{\Klat}.
$$
Consider the pull-back
\begin{equation}\label{eq:overDelta}
\begin{array}{ccccc}
  \spzero  K\persp_{\delta} \hskip -.3cm & \inj  &  K\persp_{\delta} & \inj & \hskip -.3cm H\persp_{\delta} \\
  &\phantom{{}_{\Pi_{\persp}}}\searrow\;\;\;\phantom{\Pi_{\persp}} & \mapdown &\phantom{\Pi_{\persp}}\;\;\;\swarrow\phantom{\Pi_{\persp}} & \\
  &&\phantom{,,}\Delta\;, &&
\end{array}
\end{equation}
of the diagram~\eqref{eq:overpersp} by $\delta: \Delta\inj \Omega_{(\psi\circ \sigma)\sp\perp}\inj \persp_{\Klat}$.
For simplicity, we put
$$
H:=H^{[\eta_0]}, \quad \Gamma:= \Gamma^{[\eta_0]}.
$$
Then we have trivializations
\begin{equation}\label{eq:trivs}
K\persp_{\delta}\cong \Delta\times\Gamma
\quand
 H\persp_{\delta} \cong \Delta\times H
\end{equation}
over $\Delta$ that  extend the identity maps over $0\in \Delta$,
and such that the inclusion 
$K\persp_{\delta} \inj  H\persp_{\delta} $
is given by the identity map of ${\Delta}$ times the inclusion $\Gamma\inj H$.
Since $([\per], h)\in K\persp_{\Klat}$ for any $[\per]\in \Omega_{(\psi\circ \sigma)\sp\perp}$,
we have a section 
$t\mapsto (\delta(t), h)$ of $K\persp_{\delta}\to \Delta$.
We  choose the trivialization~\eqref{eq:trivs}
in such a way that 
$K\persp_{\delta}\cong \Delta\times\Gamma$ 
maps this section to the constant section 
$t\mapsto (t, h)$ of $\Delta\times\Gamma\to \Delta$.
For a vector $d\in \Klat$ with $d^2=-2$ and a point $[\per]\in \persp_{\Klat}$ with $\intnum{\per}{d}=0$,
we put
$$
W(d):=\set{x\in \Klat\tensor\R}{\intnum{x}{d}=0}
\quand d_{[\per]}\sperp:=W(d)\cap H^{[\per]}.
$$
Then $d_{[\per]}\sperp$ is a hyperplane of the real vector space $H^{[\per]}$.
Since $\gen{\EEE}\subset\gen{\EEE_0}\sp+$, we see that $\gen{\EEE}$ is a sublattice of $\gen{\EEE_0}$,
and hence the set $D_{\gen{\EEE}}$ of roots in $\gen{\EEE}$ is 
a subset of the set $D_{\gen{\EEE_0}}$ of roots in $\gen{\EEE_0}$;
$$
D_{\gen{\EEE}}\subset D_{\gen{\EEE_0}}.
$$
We have
$$
D_{\gen{\EEE}}\subset D^{[\eta_t]}\quad\textrm{for any $t\in \Delta$,}\quand
D_{\gen{\EEE_0}}\subset D^{[\eta_0]}.
$$
More precisely, we have
\begin{equation}\label{eq:Detat}
D_{\gen{\EEE}}=\set{d\in D^{[\eta_t]}}{h\in d_{[\eta_t]}\sperp}\quad\textrm{for $t\ne0$},
\end{equation}
because $\delta(t)\in  \Omega_{(\psi\circ \sigma)\sp\perp}\sp\perspcond$ for $t\ne 0$, and 
\begin{equation}\label{eq:Deta0}
D_{\gen{\EEE_0}}=\set{d\in D^{[\eta_0]}}{h\in d_{[\eta_0]}\sperp}.
\end{equation}
We choose the trivialization~\eqref{eq:trivs}
in such a way that, for each $d\in D_{\gen{\EEE}}$,
the isomorphism $H\persp_{\delta}\cong \Delta\times H$ 
maps the family of walls
$$
\set{([\eta_t], x)\in H\persp_{\delta}}{x\in d_{[\eta_t]}\sperp}
$$
over $\Delta$ to the constant family $\Delta\times d_{[\eta_0]}\sperp$. 
We denote by
$$
\spzero(\Delta\times\Gamma) 
$$
the open subset of $\Delta\times\Gamma$ that corresponds to the open subset  
$\spzero  K\persp_{\delta} \subset   K\persp_{\delta} $
by the trivialization,
and put
$$
\WWW\;\;:=\;\;(\Delta\times\Gamma) \;\setminus\; \spzero(\Delta\times\Gamma).
$$
Recall that  the complement
of $\spzero  K\persp_{\delta}$ in $ K\persp_{\delta} $
is the union of walls
$$
\set{([\eta_t], x)\in K\persp_{\delta}}{x\in d_{[\eta_t]}\sperp \;\;\textrm{for some}\;\;  d\in D^{[\eta_t]}}.
$$
Therefore, by the description~\eqref{eq:Detat}~and~\eqref{eq:Deta0} of walls passing through  $h$, 
if $\Ball \subset\Gamma$ is a sufficiently small ball with the center $h$, then
$$
(\Delta\times\Ball)\cap\WWW=
\bigcup_{d\in D_{\gen{\EEE}}} (\Delta\times d_{[\eta_0]}\sp{\prime\perp} )  
\cup
\bigcup_{d\in D_{\gen{\EEE_0}}\setminus D_{\gen{\EEE}}} (\{0\}\times d_{[\eta_0]}\sp{\prime\perp}).
$$
where $d_{[\eta_0]}\sp{\prime\perp}:=d_{[\eta_0]}\sperp\cap \Gamma$.
In other words, the projection 
$$
\spzero(\Delta\times\Gamma) \cap (\Delta\times\Ball)\;\;\to\;\; \Delta
$$
is a constant family of cones in the ball $\Ball$ 
partitioned by the walls 
associated with $d\in  D_{\gen{\EEE}}$
over $\Delta\sptimes$, 
with the central fiber being partitioned further by the walls 
associated with $d\in D_{\gen{\EEE_0}}\setminus D_{\gen{\EEE}}$.

We have a unique connected component
of the central fiber 
$$
\spzero(\Delta\times\Gamma) \cap (\{0\}\times\Ball)\;\;\subset\;\;\{0\}\times \Gamma= \Gamma^{[\eta_0]}
$$
that is mapped  to the K\"ahler cone $\Kahler_{X_{B_0}}\subset\zeroGamma_{X_{B_0}}$
of $X_{B_0}$ via the marking $\phi_0$.
We choose a point $(0, v_0)$ from this connected component.
Then $v_0\in \Gamma^{[\eta_0]}$ corresponds to a K\"ahler class of $X_{B_0}$ via the marking $\phi_0$.
In particular, we have
\begin{equation*}\label{eq:v0EEE0}
\textrm{
$\intnum{v_0}{e}>0$ holds
for any $e\in \EEE_0$.
}
\end{equation*}
Since $\EEE\subset \gen{\EEE_0}\sp+$, we have 
\begin{equation}\label{eq:v0EEE}
\textrm{
$\intnum{v_0}{e}>0$ holds
for any $e\in \EEE$.
}
\end{equation}
By the description of $\spzero(\Delta\times\Gamma) \cap (\Delta\times\Ball)$ above, we see that
$(t, v_0)\in \Delta\times\Gamma$ is a point of $\spzero(\Delta\times\Gamma)$ for any $t\in \Delta$.
We denote by
$$
\tilde{\delta}:\Delta\to \spzero K\persp_{\delta}
$$
the section of $ K\persp_{\delta}\to \Delta$ corresponding to 
the constant section $t\mapsto (t, v_0)$ of $\spzero(\Delta\times\Gamma)\to \Delta$,
and let
$$
\tilde{\delta}_{\MMM}: \Delta\to \MMM_2
$$
be the map corresponding to $\tilde{\delta}$ via $\tau_2$.
We denote by
 $$
 (X_t, \phi_t, \kappa_t)
 $$
the marked $K3$ surface $(X_t, \phi_t)$ with a K\"ahler class $\kappa_t$
corresponding to $\tilde\delta_{\MMM}(t)\in \MMM_2$.
Let $h_{X_t}\in H^2(X_t, \Z)$ be the vector
such that $\phi_t (h_{X_t})=h$.
Since $\eta_t\perp h$, we have $h_{X_t}\in \NS(X_t)$.
Suppose that $t\ne 0$.
Since $h$ is contained in the closure of the connected component of $\zeroGamma^{[\eta_t]}$
containing $\phi_t (\kappa_t)$,
the class $h_{X_t}\in \NS(X_t)$
is nef by Proposition~\ref{prop:nefcone}.
By Proposition~\ref{prop:polarization} and $\delta(t)\in \Omega_{(\psi\circ \sigma)\sp\perp}\sp\perspcond$,
the condition~\eqref{eqcond2} in the definition of $\Omega_{(\psi\circ \sigma)\sp\perp}\sp\perspcond$
implies that  $h_{X_t}$ is the class of a polarization $\pol_t$ 
of degree $2$ on $X_t$.
Note that we have
$\intnum{\kappa_t}{e}>0$
for any $e\in \phi_t(\EEE)$  by~\eqref{eq:v0EEE}.
By $\delta(t)\in \Omega_{(\psi\circ \sigma)\sp\perp}\sp\perspcond$  again,
the condition~\eqref{eqcond1} in the definition of $\Omega_{(\psi\circ \sigma)\sp\perp}\sp\perspcond$
implies that $\phi_t\inv (\EEE)$ is  a fundamental system of roots 
in $\gen{h_{X_t}}\sp\perp\subset \NS(X_t)$
associated with the K\"ahler class $\kappa_t$.
Consequently, 
Proposition~\ref{prop:geomroots} implies that 
$\phi_t\inv (\EEE)$ is equal to the set of classes of $(-2)$-curves 
contracted by  $\Phi_{|\LLL_t|}: X_t\to \P^2$.
Let $B_t$ be the branch curve of $\Phi_{|\LLL_t|}$.
Then the markings $\phi_t: H^2(X_t, \Z)\cong \Klat$ yield 
isomorphisms of lattices from $\lat_{B_t}\subset H^2(X_t, \Z)$ to $\lat\subset \Klat$ 
 that induce  isomorphisms of
lattice data $\latdata(B_t)\cong \latdata$ for $t\ne 0$.
We define  $\alpha_t :  \lat\isom\lat_{B_t}$ to be the inverse of this isomorphism.

We will show that,
making  $\Delta$  smaller  if necessary,
these simple sextics $B_t$ form an analytic family.
Let $\pi_{\tilde{\delta}}:\XXX_{\tilde{\delta}}\to \Delta$ be the family of $X_t$,
which is the pull-back of the universal family $\pi_1: \XXX_1\to \MMM_1$
by $\Pi_{\MMM}\circ\tilde{\delta}_{\MMM}$.
Then $t\mapsto h_{X_t}$ gives a section of $R^2 \pi_{\tilde{\delta}*}\Z$.
By Lemma~\ref{lem:LLL}, 
there exists a line bundle $\pol_{\XXX}$ on $\XXX_{\tilde{\delta}}$
such that 
the restriction $\pol_{\XXX}|X_t$ is
equal to the polarization $ \pol_t$ given above for any $t\in \Delta$.
Note that $h^0(X_0, \pol_0)=3$ and $h^1(X_0, \pol_0)=0$ by  Nikulin~\cite[Proposition 0.1]{MR1260944}.
Therefore, 
shrinking  $\Delta$  if necessary,
we have  a $3$-dimensional subspace $V$ of $H^0 (\XXX_{\tilde{\delta}}, \pol_{\XXX})$
such that the restriction homomorphism  maps $V$ onto $H^0 (X_t, \pol_t)$ 
isomorphically  
for any $t\in \Delta$.
In particular, the linear system $V$ has not base points.
Considering the morphism 
$$
\Phi_V: \XXX_{\tilde{\delta}}\to \Pt
$$
induced by $V$,
we obtain an analytic  family of morphisms $X_t\to \P^2$ with the branch curve  $B_t\subset \Pt$,
and hence we obtain an  analytic family of simple sextics over $\Delta$.
It is obvious that this analytic family and the isomorphisms $\alpha_t :  \lat\isom\lat_{B_t}$
of lattice data from $\latdata$ to $\latdata(B_t)$ for $t\ne 0$ have the required properties.
\end{proof}
By Proposition~\ref{prop:onlyifsp} together with the construction of 
the geometric embedding $\sigma_{\BBB, t}$, we obtain the following:
\begin{corollary}\label{cor:ifsp}
Let $\lattype_0$ and $\lattype$ be lattice types of simple sextics,
and let  $\latdata_0$ and $\latdata$ be the corresponding lattice data. 
Then  $\lattype_0$ is a specialization of $\lattype$ if and only if  
there exists a geometric embedding  of $\latdata$ into $\latdata_0$.
\end{corollary} 
\begin{remark}
By the theory of adjacency of singularities (\cite{MR0397777} or~\cite{MR0350775}),
we see that,  if $\lattype(B_0)$ is a specialization of $\lattype(B)$,
the Dynkin diagram of $R_B$ is a subgraph of the Dynkin diagram of $R_{B_0}$.
\end{remark}
Let $B$ be a simple sextic, and let
$D:=C_1+\dots +C_m$
be an effective divisor on $X_B$,
where $C_1, \dots, C_m$ are reduced and irreducible.
A \emph{subcurve} of $D$ is, by definition,  a divisor
$$
C:=C_{i_1}+\dots+ C_{i_n},
$$
where $\{C_{i_1}, \dots, C_{i_n}\}$ is a (possibly empty) subset of $\{C_1, \dots,C_m\}$.
%
\begin{lemma}\label{lem:h1}
Let
$D:=C_1+\dots +C_m$
be an effective divisor on $X_B$.
We put $h^1(D):=\dim H^1(\XB, \OOO(D))$.


{\rm (1)}
Suppose  that $D^2=-2$.
For  $h^1(D)=0$ to hold,  it is sufficient  that 
$C^2\le -2$  holds for  any non-empty subcurve $C$ of $D$.

{\rm (2)}
Suppose  that $D^2=0$ and $\intnum{D}{h_B}=3$.
For $h^1(D)=0$ to hold,  it is sufficient  that 
$C^2\le 0$  holds for  any non-empty subcurve $C$ of $D$.
\end{lemma}
\begin{proof}
Let $|M|$ be the movable part of $|D|$,
where $M$ is a subcurve of $D$.
Suppose that $D^2=-2$.
If $h^1(D)>0$, then we have $|M|\ne \emptyset$ and hence $M^2\ge 0$.
Suppose that $D^2=0$ and $\intnum{D}{h_B}=3$.
If $h^1(D)>0$,
then either $M^2>0$ or $|M|=m|E|$ with $m>1$ for some elliptic pencil $|E|$.
Since $\intnum{D}{h_B}=3$, we would have  $\intnum{E}{h_B}=1$ in the latter case,
which is absurd.
\end{proof}
We interpret this geometric fact into a lattice-theoretic sufficient condition,
which can be checked easily by a computer.
\begin{definition}
Let $\latdataZP=[\EEE, h, \lat, \{v\sp\pm\}]$
be  extended lattice data,
and let 
$$
w:=v^+ +\sum  m_e e\qquad (e\in \EEE, \;m_e\ge 0)
$$
be an element of $v^++\gen{\EEE}^+$.
We say that $u\in \lat$ is a \emph{subcurve vector} of $w$ 
if $u$ is  $n_v v^+ +\sum  n_e e$ with $m_e\ge n_e \ge 0$ for any $e\in \EEE$ and ($n_v=0$ or $n_v=1$).

Suppose that $ w^2=-2$ or ($w^2= 0$ and $\intnum{w}{h}=3$).
We say that \emph{$w$ satisfies the vanishing-$h^1$ condition}
if $u^2\le w^2$ holds for any non-zero subcurve vector $u$ of $w$.

We also define the vanishing-$h^1$ condition  for  elements $w$ of $v^-+\gen{\EEE}^+$
in the same way.
\end{definition}
\begin{definition}
We say that a geometric embedding 
$\sigma$ of $\latdataZP=[\EEE, h, \lat, \{v\sp\pm\}]$ into $\latdataZP_0=[\EEE_0, h_0, \lat_0, \{v\sp\pm_0\}]$
satisfies the vanishing-$h^1$ condition
if  $\sigma (v^+)\in \{v_0\sp\pm\}+\gen{\EEE_0}\sp+$ satisfies the vanishing-$h^1$ condition.
\end{definition}
\begin{proposition}\label{prop:onlyifspZP}
Let $\latdataZP=[\EEE, h, \lat, \{v\sp\pm\}]$ and $\latdataZP_0=[\EEE_0, h_0, \lat_0, \{v\sp\pm_0\}]$ be
the lattice data of  $Z$-splitting pairs
$(B, \Gm)$ and $(B_0, \Gm_0)$, respectively.
Suppose that $\Gm$ and  $\Gm_0$ are smooth of degree $\le 3$.
Then the lattice type $\lattypeZP(B_0, \Gm_0)$ is a specialization of the lattice type $\lattypeZP(B, \Gm)$
if  there exists a geometric embedding
$\sigma: \lat\inj \lat_{0}$ 
of $\latdataZP$ into $\latdataZP_0$ 
that satisfies  the vanishing-$h^1$ condition.
\end{proposition}
\begin{proof}
By Remark~\ref{rem:latticegenericZP},
we can assume that the representatives
$(B, \Gm)$ and $(B_0, \Gm_0)$ of $\lattypeZP(B, \Gm)$ and $\lattypeZP(B_0, \Gm_0)$
are lattice-generic. 
We  fix a marking
$$
\phi_0: H^2(X_{B_0},\Z)\isom \Klat.
$$
We then  consider  $\lat_0$ as a primitive sublattice of $\Klat$
in such a way that the marking  $\phi_0$ induces an isomorphism 
$$
\phi_0: \lat_{B_0}\isom \lat_0
$$
of lattice data from $\latdataZP(B_0)$ to $\latdataZP_0$.
By Proposition~\ref{prop:onlyifsp}, we have 
an analytic family $\{B_t\}_{t\in \Delta}$ of simple sextics 
constructed from the geometric embedding
$\sigma: \lat\inj \lat_{0}$ 
of $\latdata=[\EEE, h, \lat]$ into $\latdata_0=[\EEE_0, h_0, \lat_0]$ and the isomorphism $\phi_0$.
Let 
$$
\pi_{\tilde\delta}:\XXX_{\tilde\delta}\to\Delta
$$
be the smooth family of $K3$ surfaces 
constructed in the proof of Proposition~\ref{prop:onlyifsp}.
Then $X_t:=\pi_{\tilde\delta}\inv (t)$ is equal to $X_{B_t}$
and equipped with markings 
$$
\phi_t : H^2(X_t, \Z)\isom \Klat
$$
continuously varying with $t$.
We have lifts $\tlGm_0\sp\pm$ of $\Gm_0$ on $X_0=X_{B_0}$.
Our aim is to deform $\tlGm_0\sp\pm$ to curves on $X_t$
that are the lifts of $Z$-splitting curves  for $B_t$.
 \par
 \medskip
By construction, 
 the markings $\phi_t$ induce isomorphisms of lattices
$$
\phi_t: \lat_{B_t}\isom \lat
$$
for $t\ne 0$ that 
induces an isomorphism of lattice data $\latdata(B_t)\cong \latdata$.
Moreover the specialization  homomorphism 
$$
H^2(X_t, \Z)\isom H^2(X_0, \Z)
$$
induces the geometric embedding $\sigma:\lat\inj\lat_0$ of $\latdata$ to $\latdata_0$ under the isomorphisms 
$\phi_t$  $(t\ne 0)$ and $\phi_0$.
Then $v\sp+ \in\lat$ with $\sigma (v\sp +)\in \lat_0$
gives rise to a section $\tilde v$ of the locally constant system $R^2 \pi_{\tilde\delta*}\Z$ on $\Delta$;
namely, 
 $\tilde v_t:=\tilde v|X_t \in H^2(X_t, \Z)$ is mapped by $\phi_t$ to
$v\sp+$ for $t\ne 0$ and 
to $\sigma (v\sp +)$ for $t=0$.
In particular, we have $\tilde v_t\in H^{1,1} (X_t)$ for any $t\in \Delta$,
and hence, by Lemma~\ref{lem:LLL},  there exists a line bundle $\DDD$ on $\XXX_{\tilde\delta}$
such that the class of $\DDD_t:=\DDD|X_t$ is equal to $\tilde v_t$.
Since $[\EEE, h, \lat, \{v\sp\pm\}]$ is the lattice data of 
$(B, \Gm)$,
the  assumption that $\Gm$ be smooth of degree $\le 3$ implies that
$v\sp+$ satisfies 
$ (v\sp+)^2=-2$ or ($(v\sp+)^2= 0$ and $\intnum{v\sp+}{h}=3$),
and hence $\sigma (v\sp+)\in \lat_0$ also satisfies 
$$
(\sigma (v\sp+))^2=-2 \quad \textrm{or} \quad (\;(\sigma (v\sp+))^2= 0 \quand \intnum{\sigma (v\sp+)}{h_0}=3\;).
$$
Therefore Lemma~\ref{lem:h1} can be applied,
and the assumption that $\sigma (v\sp+)$ satisfy
the vanishing-$h^1$ condition  implies
$$
H^1(X_0, \DDD_0)=0.
$$
After interchanging $v_0^+$ and $v_0^-$ (and hence $\tlGm_0^+$ and $\tlGm_0^-$) if necessary,
there exist a finite number of  exceptional $(-2)$-curves 
$e_i$ on $X_0$ such that
$$
\tilde v_0=[\tlGm_0^+ + \sum e_i].
$$
Let $s_0$ be the section of the invertible sheaf $\OOO(\tlGm_0^+ + \sum e_i)$ on $X_0$
such that $s_0=0$ defines the divisor $\tlGm_0^+ + \sum e_i$.
By Lemma~\ref{lem:V},
there exists a section $s\in H^0(\XXX_{\tilde\delta}, \DDD)$
such that its restriction to $X_0$ is $s_0$.
We put $s_t:=s|X_t$ for $t\ne 0$,
and let $\tlGm_t$ be the curve on $X_t$ cut out by $s_t=0$.
Since $\phi_t ([\tlGm_t])= v^+\in \lat$,
we have $[\tlGm_t]\in \lat_{B_t}$.
Since $[\EEE, h, \lat, \{v\sp\pm\}]$ is   the  lattice data of 
$(B, \Gm)$ and $[\tlGm\sp\pm]\in\ZZZ_n(B)$ with $n=\deg\Gm=\intnum{v\sp+}{h}\le 3$, 
we see that,  if $B_{\tau}$ is lattice-generic with $\tau\ne 0$,
then
$$
[\tlGm_{\tau}]\in \ZZZ_n (B_{\tau})
$$
holds by Theorem~\ref{thm:ZZZ}.
In particular, if $n<3$, then
$\tlGm_t$ is a $(-2)$-curve.
When $n=3$, we replace $s$ by $s+s\sprime$
where 
$$
s\sprime\in H^0(\XXX_{\tilde\delta}, \DDD(-X_0))=H^0(\XXX_{\tilde\delta}, \DDD)\otimes \OOO_{\Delta}(-0)
$$
is chosen generally,
and assume that $\tlGm_t$ is irreducible.
We denote by $\Gm_t$ the image of $\tlGm_t$ by the double covering
$X_{B_{t}}\to\Pt$.
Then $\Gm_t$ is  a smooth $Z$-splitting curve that  degenerates to $\Gm_0$.
Since the lattice data of $(B_t, \Gm_t)$ for $t\ne 0$ is isomorphic to $\latdataZP$,
the analytic family $(B_t, \Gm_t)_{t\in \Delta}$ of $Z$-splitting pairs 
gives rise to the specialization of $\latdataZP$ to $\latdataZP_0$.
\end{proof}
%
%
\begin{computation}\label{comp:specialization}
By Computation~\ref{comp:latdata},
we have obtained  the complete list $\LD_n$  of lattice data 
of $Z$-splitting pairs $(B, \Gm)$ with $n:=\deg\Gm\le 2$,
and the complete list $\LD_3$ of lattice data  
of $Z$-splitting pairs $(B, \Gm)$ with $z_1(\lattype(B))=z_2(\lattype(B))=0$, $F_B\ne 0$ 
and $\Gm$ being smooth cubic.

For each $\latdataZP=[\EEE, h, \lat, S]$ in $\LD_1$ (resp.~$\LD_2$),
we calculate the class-order $d$  of $\latdataZP$ (that is, the order of $v\in S$ in 
the finite abelian group $\lat/(\gen{h}\oplus\gen{\EEE})$),
and confirm that $d$ is either $6,8,10$ or $12$ (resp.~$3,4,5,6,7$ or $8$).

For each $n=1, 2$ and the class-order $d$,
we denote by $\LD_{n, d}$ the set of  lattice data $\latdataZP\in \LD_n$  with the class-order $d$, and 
denote by $\ttlatdataZP_{n, d}$ the member of $\LD_{n, d}$ with  the total Milnor number $\mu_B=\rank \gen{\EEE}$ being \emph{minimal}.
It turns out that the condition that  $\mu_B$ be minimal determines
$\ttlatdataZP_{n, d}$ uniquely, and that the corresponding lattice types are 
equal to $\lattypeZP_{lin, d}$ or $\lattypeZP_{con, d}$ given in Definitions~\ref{def:Zlineslattype2}~or~\ref{def:Zconicslattype2}
according to $n=1$ or $2$.
Then,  for each $\latdataZP$ in  $\LD_{n, d}$ that is not $\ttlatdataZP_{n, d}$,
we search for a geometric embedding  of $\ttlatdataZP_{n, d}$ into $\latdataZP$
that satisfies the vanishing-$h^1$ condition,
and confirm that there exists at least one such embedding.
Thus Theorems~\ref{thm:Zlines} and~\ref{thm:Zconics} are proved.

We also confirm that
there exists  unique lattice data $\ttlatdataZP_{QC}$ in $\LD_3$ with $\mu_B$ being minimal,
that the  lattice type corresponding to $\ttlatdataZP_{QC}$ is $\lattype_{QC, n}$, 
and that,
for each piece of lattice data $\latdataZP$ in $\LD_3$ that is not $\ttlatdataZP_{QC}$, 
there exists at least one  geometric embedding  of $\ttlatdataZP_{QC}$ into $\latdataZP$
 satisfying the vanishing-$h^1$ condition.
Thus the second half of Theorem~\ref{thm:sp} is also proved.
\end{computation}
\section{Demonstration}\label{sec:demonstration}
We demonstrate the calculations for the $ADE$-type
$A_3+2A_7$.
 Let $\gen{\EEE}$ be the negative-definite root lattice of type $A_3+2A_7$
with a distinguished fundamental system of roots
$$
\EEE\;=\;\{t_1, t_2, t_3\}\:\perp \:    \{e_1, \dots, e_7 \}\:\perp \:      \{e\sprime_1, \dots, e\sprime_7\},
$$
where $\{t_1, t_2, t_3\}$ is of type $A_3$ 
with $\intnum{t_i}{t_{i+1}}=1$ for $i=1, 2$,
 and $\{e_1, \dots, e_7\}$
and  $\{e\sprime_1, \dots, e\sprime_7\}$ are of type $A_7$
with $\intnum{e_i}{e_{i+1}}=\intnum{e\sprime_i}{e\sprime_{i+1}}=1$ for $i=1, \dots, 6$.
The automorphism group $\Aut (\EEE)$ of $\EEE$ is isomorphic to
$$
\{\pm 1\} \times  (\{\pm 1\}\wr \SSSS_2),
$$
where the first factor is the involution $t_1\leftrightarrow t_3$ of $A_3$,
and $\{\pm 1\}\wr \SSSS_2$ is the wreath product of the involution $e_i \leftrightarrow e_{8-i}$
of $A_7$ and the permutation of the components of $2A_7$.
We put
$$
\blat=   \gen{\EEE}\oplus \gen{h},
$$
where $h^2=2$.
Then the discriminant group $\blat\dual/\blat$
of $\blat$ 
is 
\begin{equation*}\label{eq:discgroup}
\gen{\bar t_3\dual}\oplus  \gen{\bar e_7\dual}\oplus \gen{\bar e\sp{\prime\vee}_7} 
\oplus \gen{\bar h\dual}
\;\;\cong\;\;
(\Z/4\Z)\oplus (\Z/8\Z)\oplus (\Z/8\Z) \oplus (\Z/2\Z),
\end{equation*}
where $\bar x=x\bmod \blat$,
and the discriminant form $q: \blat\dual/\blat \to \Q/2\Z$ of $\blat$ is given by
$$
q(w,x,y,z)= -\frac{3}{4}w^2-\frac{7}{8}x^2-\frac{7}{8}y^2+ \frac{1}{2}z^2\bmod 2\Z,
$$
where $(w,x,y,z)=w \bar t_3\dual+ x \bar e_7\dual + y \bar e\sp{\prime\vee}_7 + z \bar h\dual$.
We determine all isotropic subgroups $H$ such that
the corresponding overlattice $\lat=\lat(H)$ satisfies the  three conditions in Proposition~\ref{prop:urabe}.
Up to the action of $\Aut(\EEE)$, 
they are given in Table~\ref{table:H}.
\begin{table}
$$
 \renewcommand{\arraystretch}{1.2}
 \begin{array}{c|l|l}
 &\textrm{Generators} &  \\ 
\hline   
H_{0} & 0 & 0 \\ 
\hline 
H_{1} & [[0, 4, 4, 0]] & \textrm{cyclic of order $2$} \\ 
\hline
H_{2} & [[1, 1, 1, 1]] & \textrm{cyclic of order $8$}  \\ 
 \hline
 H_{3} & [[2, 2, 2, 0]] & \textrm{cyclic of order $4$} 
 \end{array}
 $$
\vskip 10pt
\caption{The isotropic subgroups $H_i$}\label{table:H}
\end{table} 
Therefore there exist four lattice types 
$\lattype (H_i)$ of simple sextics $B$
with $R_B=A_3+2A_7$.
We denote by $B(H_i)$ a lattice-generic member of $\lattype  (H_i)$.

Next  we calculate the subsets $\LLL(H_i):=\LLL_{B(H_i)}$ and $\CCC(H_i):=\CCC_{B(H_i)}$ of $\lat(H_i)$ 
for each $H_i$,
and deduce information about the geometry of $B(H_i)$.
From now on,
vectors in $\lat(H_i)\subset \Sigma\dual$ are written with respect to the basis 
$$
t_1\dual, \dots, t_3\dual, e_1\dual, \dots,  e_7\dual, e_1\sp{\prime\vee}, \dots, e_7\sp{\prime\vee}, h\dual
$$
of 
$\Sigma\dual$
that is \emph{dual} to $\EEE\cup\{h\}$.
\par\smallskip
($H_0$) We have 
$\LLL(H_0)=\emptyset$ and $\CCC(H_0)=\emptyset$.
Hence $B(H_0)$ is irreducible.
(If $\degs B(H_0)=[3,3]$, then the two cubic irreducible components would intersect with multiplicity $10$.)
Moreover we have $z_1(\lattype(H_0))=z_2(\lattype(H_0))=0$.
\par\smallskip
($H_1$) We have $\LLL(H_1)=\emptyset$ and $\CCC(H_1)=\{ u\}$, where
$$
u:=[0, 0, 0, 0, 0, 0, 1, 0, 0, 0, 0, 0, 0, 1, 0, 0, 0, 2]. 
$$
Since $u$ is invariant under the involution on $\lat(H_1)$, 
we have  $\degs B(H_1)=[2,4]$ with the  irreducible component of degree $2$ passing through two $A_7$ points and disjoint from the tacnode $A_3$.
Moreover we have $z_1(\lattype(H_1))=z_2(\lattype(H_1))=0$.
This lattice type is denoted by $\lattype_{\smallBBBB, n}$ in Proposition~\ref{prop:Zlineslattype1}.
\par\smallskip
($H_2$) We have $\LLL(H_2)=\{v, \iota_B(v)\}$ and $\CCC(H_2)=\{ u\}$, where
$$
v:=[1, 0, 0, 1, 0, 0, 0, 0, 0, 0, 1, 0, 0, 0, 0, 0, 0, 1] \;\ne\; \iota_B(v),
$$
and $u$ is the same vector as in ($H_1$).
Hence we have $ B(H_2)\relconfig B(H_1)$, and 
$z_1(\lattype(H_2))=1, z_2(\lattype(H_2))=0$.
This lattice type is denoted by $\lattype_{\smallBBBB, l}$.
The class $v$ of the lift of  $Z$-splitting line 
is  of order $8$ in the discriminant group $\blat\dual/\blat$.
There are no $Z$-splitting lines of class-order $8$ for simple sextics of total Milnor number $<17$.
Hence the $Z$-splitting line for $B(H_2)$ is the originator
of the lineage of $Z$-splitting lines of class-order $8$,
whose lattice type is denoted by $\lattypeZP_{lin, 8}$.
\par\smallskip
($H_3$) We have $\LLL(H_3)=\emptyset$ and $\CCC(H_3)=\{ u, w, \iota_B(w)\}$, where
$$
w:=[0, 1, 0, 0, 1, 0, 0, 0, 0, 0, 0, 1, 0, 0, 0, 0, 0, 2]  \;\ne\; \iota_B(w), 
$$
and $u$ is the same vector as in ($H_1$).
Hence we have $ B(H_3)\relconfig B(H_1)$, and 
$z_1(\lattype(H_3))=0, z_2(\lattype(H_3))=1$.
This lattice type is denoted by $\lattype_{\smallBBBB, c}$.
The class $w=[\tlGm] $ of the lift of  $Z$-splitting conic $\Gm$
is  of order $4$ in the discriminant group $\blat\dual/\blat$.
The conic $\Gm$ is tangent to the quartic irreducible component of $B(H_3)$
at the three singular points of $B(H_3)$.
\par\medskip
Next we 
describe the originator of 
 the lineage of $Z$-splitting conics of class-order $4$,
 and how the $Z$-splitting conic for $ B(H_3)$ above is obtained 
 from this originator by specialization.
\par\smallskip
Any simple sextic of
total Milnor number $<14$
does not have  $Z$-splitting conics of class-order $4$,
and there exists a unique lattice type $\lattype_{\smallbbbb, c}$
of total Milnor number $14$
whose lattice-generic member $B\sprime$ has  a $Z$-splitting conic $\Gm$  of class-order $4$.
The $ADE$-type of the lattice type is $2A_1+4A_3$.
Consider the negative-definite root lattice $\gen{\EEE\sprime}$ 
of type $2A_1+4A_3$ with a distinguished fundamental system of roots
$$
\EEE\sprime:=\{a^{(1)}\}\perp\{a^{(2)}\}\perp\{b^{(1)}, c^{(1)}, d^{(1)}\}\perp\dots \perp\{b^{(4)}, c^{(4)}, d^{(4)}\},
$$
where $\{a^{(\nu)}\}$ is of type $A_1$ and 
$\{b^{(\nu)}, c^{(\nu)}, d^{(\nu)}\}$  is of type $A_3$
with $\intnum{b^{(\nu)}}{c^{(\nu)}}=\intnum{c^{(\nu)}}{d^{(\nu)}}=1$.
We put
$$
\blat\sprime= \gen{\EEE\sprime} \oplus \gen{h}  .
$$
The discriminant group  of $\blat\sprime$ is 
isomorphic to 
$$
(\Z/2\Z)^2 \oplus (\Z/4\Z)^4 \oplus (\Z/2\Z),
$$
with 
$$
q\sprime (x_1, x_2, y_1, y_2, y_3, y_4, z)=
-\frac{1}{2}x_1^2-\frac{1}{2}x_2^2-
\frac{3}{4}y_1^2-\frac{3}{4}y_2^2-\frac{3}{4}y_3^2-\frac{3}{4}y_4^2+ \frac{1}{2}z^2\bmod 2\Z.
$$
The overlattice $\lat_{B\sprime}$ of the lattice type $\lattype_{\smallbbbb, c}$
corresponds to the isotropic subgroup
$$
H\sprime:=\gen{[1, 1, 1, 1, 1, 1, 0]},
$$
which is cyclic of order $4$.
We denote vectors of $\lat_{B\sprime}\subset (\blat\sprime)\dual$ 
with respect to the basis of $(\blat\sprime)\dual$ dual to  
the basis $\EEE\sprime\cup\{h\}$ of $\blat\sprime$.
Then the classes of the lifts of the $Z$-splitting conic $\Gm\sprime$
for the lattice-generic member $B\sprime$ of $\lattype_{\smallbbbb, c}$ are
equal to 
$$
w\sprime:=
[1, 1, 0, 0, 1, 0, 0, 1, 0, 0, 1, 0, 0, 1, 2]
$$
and $\iota_B(w\sprime)$.
Let $\sigma: (\blat\sprime)\dual \to \blat\dual$ be the homomorphism
given by the matrix 
{\scriptsize
$$
 \left[ \begin {array}{ccccccccccccccc} 
 1&0&0&0&0&0&0&0&0&0&0&0&0&0&0
\\-a&-a&0&0&0&0&0&0&0&0&0&0&0&0&0
\\0&1&0&0&0&0&0&0&0&0&0&0&0&0&0
\\0&0&0&0&0&-b&-a&-c&b&a&c&0&0&0&0
\\0&0&0&0&0&0&0&1&0&0&0&0&0&0&0
\\0&0&0&0&0&0&1&0&0&0&0&0&0&0&0
\\0&0&0&0&0&b&-a&-b&-b&-a&b&0&0&0&0
\\0&0&0&0&0&0&0&0&0&1&0&0&0&0&0
\\0&0&0&0&0&c&a&b&b&-a&-b&0&0&0&0
\\0&0&0&0&0&-c&-a&-b&c&a&b&0&0&0&0
\\0&0&0&0&1&0&0&0&0&0&0&0&0&0&0
\\0&0&b&a&-b&0&0&0&0&0&0&-b&-a&-c&0
\\0&0&0&0&0&0&0&0&0&0&0&0&0&1&0
\\0&0&-b&a&b&0&0&0&0&0&0&b&a&-b&0
\\0&0&1&0&0&0&0&0&0&0&0&0&0&0&0
\\0&0&-c&-a&-b&0&0&0&0&0&0&-b&a&b&0
\\0&0&0&0&0&0&0&0&0&0&0&1&0&0&0
\\0&0&0&0&0&0&0&0&0&0&0&0&0&0&1
\end {array} \right],
$$}%
where $a=1/2$, $b=1/4$ and $c=3/4$.
It can be easily checked 
that $\sigma (h)=h$, $\sigma (\EEE\sprime) \subset \gen{\EEE}^+$,
that $\sigma$ embeds the lattice $\lat_{B\sprime}\subset (\blat\sprime)\dual $
into the lattice $\lat_{B(H_3)} \subset \blat\dual$ primitively.
Moreover, we have
$$
\sigma (w\sprime)=w+t_2+e_2\sprime.
$$
We can easily see  that $\sigma (w\sprime)=w+t_2+e_2\sprime$
satisfies the vanishing-$h^1$ condition.
Therefore $\lattypeZP(B(H_3), \Gm)$ is a specialization of
$\lattypeZP(B\sprime, \Gm\sprime)$.
\begin{remark}
There are six configuration types
and seven lattice types  with $ADE$-type $2A_1+4A_3$.
\end{remark}
\begin{remark}
This triple 
$\{\lattype_{\smallBBBB, c}, \lattype_{\smallBBBB, l},  \lattype_{\smallBBBB, n}\}$
 is the example of  lattice Zariski triple
 with the smallest total Milnor number.
\end{remark}
\begin{remark}
Let $B_\tau$ be a sextic in the lattice type $\lattype_{\smallBBBB, \tau}$,
where $\tau=c,l,n$,
and let $B_\tau=C_\tau\cup Q_\tau$ be the irreducible decomposition
of $B_\tau$ with $\deg Q_\tau=4$.
Consider the normalization
$$
\nu : \tilde{Q}_\tau \to Q_\tau
$$
of the quartic curve $Q_\tau$ with one  tacnode.
Then $\tilde{Q}_\tau$ is a curve of genus $1$.
Let $p, q\in \tilde{Q}_\tau$ be the inverse images  of the tacnode,
and let $s, t\in \tilde{Q}_\tau$ be the inverse images  of the two $A_7$-singular points $C_\tau\cap Q_\tau$.
Then,
in the elliptic curve $\Pic^0 (\tilde{Q}_\tau)$,
the order of the class of the divisor
$p+q-s-t$ on $\tilde{Q}_\tau$ is $4$, $2$ or $1$ according to 
$\tau=c,l$ or $n$.
\end{remark}
\section{Miscellaneous facts and final remarks}\label{sec:remarks}
\subsection{Numerical criterion of the pre-$Z$-splittingness}\label{subsec:criterion}
%
%
%
\begin{definition}
Let $\Gm$ be a smooth splitting curve for $B$ that is not contained in $B$.
Let  $P$ be a singular point of $B$.
We define $\sigma_P (\Gm)\in \Q$ as follows.
If $P\notin \Gm$,
we put $\sigma_P (\Gm):=0$.
Suppose that  $P\in \Gm$.
If $P$ is of type $A_l$, then 
$$
\sigma_P (\Gm):= -m^2/(l+1),\;\;\textrm{where}\;\; m=\min(\tau_P(\tlGmplus), l+1-\tau_P(\tlGmplus)).
$$
(Recall that  $\tau_P(\tlGmplus)$ is defined in Definition~\ref{def:tau}.)
If $P$ is of type $D_m$, then 
$$
\sigma_P (\Gm):=\begin{cases}
-m/4 & \text{if $m$ is even and $\tau_P(\tlGmplus)=1$ or $2$}, \\
1/2-m/4 & \text{if $m$ is odd and $\tau_P(\tlGmplus)=1$ or $2$}, \\
\tau_P(\tlGmplus)-m-1 & \text{if $\tau_P(\tlGmplus)\ge 3$}.
\end{cases}
$$
If $P$ is of type $E_n$, then 
$\sigma_P (\Gm)$ is defined by the following table:
$$
\renewcommand{\arraystretch}{1.2}
\begin{array}{c|ccccccccc}
\tau_P(\tlGmplus) & 1 & 2 & 3 & 4 & 5 & 6 & 7 & 8 \\
\hline
E_6 & -2 &-2/3 &-8/3 &-6 &-8/3 &-2/3 &&\\
E_7 & -7/2 &-2 &-6 &-12 &-15/2 &-4 &-3/2 &\\
E_8 & -8 &-4 &-14 &-30 &-20 &-12 &-6 &-2 
\end{array}
$$
We can easily check  that $\sigma_P (\Gm)$ does not depend on the choice of the lift $\tlGmplus$
by Remark~\ref{rem:iota}.
\end{definition}
\begin{proposition}\label{prop:criterion}
Let $\tilde{B}\subset \XB$ be the reduced part of
the strict transform of $B$.
Suppose that  $\Gm$ is a smooth splitting curve for $B$ not contained in  $B$.
We put
$$
t_\Gm:=\intnum{\tilde{B}}{ \tlGmplus}=\intnum{\tilde{B}}{\tlGmminus}.
$$
Then the following inequality holds:
\begin{equation}\label{eq:ineq}
(\deg \Gm)^2/2+\tsum_{P} \sigma_P(\Gm)\le t_\Gm.
\end{equation}
The splitting curve $\Gm$ is pre-$Z$-splitting if and only if 
the equality holds in~\eqref{eq:ineq}.
\end{proposition}
\begin{proof}
Let $N_{\Q}$ denote the orthogonal complement of the  subspace $\blatB\otimes \Q=\latB\otimes \Q$
in $\NS(X_B)\otimes\Q$.
Then the intersection-paring is negative-definite on $N_{\Q}$,
and  the involution ${\iota_B}$ on $\NS(X_B)\tensor \Q$  acts on $N_{\Q}$ by the  multiplication by $-1$.
We have a decomposition
$$
[\tlGmplus]=(\deg \Gm /2) h+\sum \gamma_P +n,
$$
where $\gamma_P\in \gen{\EEE_P}\otimes\Q$ and $n\in N_{\Q}$.
Then we have
$$
t_{\Gm}=\intnum{[\tlGmplus]}{[\tlGmminus]}=(\deg \Gm)^2/2+\sum \intnum{\gamma_P }{ {\iota_B}(\gamma_P)}-n^2
$$
by Lemma~\ref{lem:tGamma}.
The value $\sigma_P (\Gm)$ is defined in such a way that 
$\sigma_P (\Gm)= \intnum{\gamma_P }{ {\iota_B}(\gamma_P)}$
holds.
Since $n^2\le 0$ and $n^2=0$  holds if and only if  $n=0$,
we obtain the proof.
\end{proof}
\begin{example}\label{example:numtorus}
Let $f$ and $g$ be general homogeneous polynomials of degree $2$ and $3$,
respectively.
The splitting conic $\Gamma=\{f=0\}$ 
for a  torus sextic $B_{\torus}=\{f^3+g^2=0\}$ is $Z$-splitting,
because we have $\deg \Gm=2$, $t_\Gm=0$ and $\sigma_P(\Gm)=-1/3$
for each ordinary cusp $P$ of $B_{\torus}$.
\end{example}
\begin{remark}
As a corollary of the classifications of $Z$-splitting pairs,
we obtain the following.
Let $(B, \Gm)$ be a lattice-generic $Z$-splitting pair with $\deg \Gm\le 2$.
Then $B\cap \Gm$ is contained in $\Sing B$,
and $\tlGm_+\cap \tlGm_-=\emptyset$.
\end{remark}
\subsection{Relation between$\relemb$ and $\rellat$}\label{subsec:lat_and_emb}
In  many  lattice Zariski $k$-ples,
the distinct lattice types have different embedding topology.
\begin{theorem}\label{thm:lat_and_emb}
Suppose that $B$ and $B\sprime$ satisfy $B\relconfig B\sprime$.
If $G_B$ and $G_{B\sprime}$ have different orders, then
$B\nrelemb B\sprime$.
\end{theorem}
\begin{proof}
We can assume that $B$ and $B\sprime$ are lattice-generic.
We consider the \emph{transcendental lattices} of $X_B$ and $X_{B\sprime}$ defined by
$$
T_B:=(\NS(X_B)\inj H^2(X_B, \Z))\sperp,
\quad
T_{B\sprime}:=(\NS(X_{B\sprime})\inj H^2(X_{B\sprime}, \Z))\sperp.
$$
From $B\relconfig B\sprime$,
we have  $R_B=R_{B\sprime}$,
and hence
$\disc \blatB=\disc \blat_{B\sprime}$ holds,
where $\disc$ denotes the discriminant of the lattice.
Combining this with $|G_B|\ne |G_{B\sprime}|$,
we obtain
$\disc \latB\ne \disc \lat_{B\sprime}$.
Since $H^2(X_B, \Z)$ and $H^2(X_{B\sprime}, \Z)$ are unimodular, we obtain
$$
\disc T_B \ne \disc T_{B\sprime}.
$$
Then $B\nrelemb B\sprime$
follows from the fact that the transcendental lattice of $X_B$ is a topological invariant of $(\Pt, B)$
for a lattice-generic $B$,
which was proved in~\cite{MR2405237} and~\cite{nonhomeo}. 
\end{proof}
\begin{remark}
We have not yet obtained any  examples of pairs  $[B_1, B_2]$ of simple sextics
with $B_1\nrellat B_2$ but $B_1\relemb B_2$.
For the example of the lattice Zariski couple 
$\lattype_{\QC, c}$ and $\lattype_{\QC, n}$ in Proposition~\ref{prop:sp},
we have $|G_B|=|G_{B\sprime}|=4$,
where $B\in \lattype_{\QC, c}$ and $B\sprime \in \lattype_{\QC, n}$,
and hence Theorem~\ref{thm:lat_and_emb} does not apply.
It would be an interesting problem to study the topology of simple sextics
in $\lattype_{\QC, c}$ and $\lattype_{\QC, n}$.
\end{remark}
\subsection{Examples of many $Z$-splitting conics}
For any lattice type $\lattype (B)$ of simple sextics,
we have $z_1(\lattype (B))\le 1$.
On the other hand,
we have lattice types $\lattype (B)$ of simple sextics
such that $z_2(\lattype (B))=12$ or $z_2(\lattype (B))=6$.
(These two  are the largest and the second largest values for $z_2(\lattype (B))$.)
\par
\medskip
Suppose that $z_2(\lattype (B))=12$.
Then $B$ is a nine cuspidal sextic.
The configuration type of nine cuspidal sextics $B$ consists 
of a single lattice type,
and  the group $G_B$ is isomorphic to $\Z/3\Z\times \Z/3\Z \times \Z/3\Z$.
Moreover the class orders of the twelve $Z$-splitting conics for $B$ are all  $3$.
A  nine cuspidal sextic $B$ is the dual curve of a smooth cubic curve $C$,
and the nine cusps  are in one-to-one correspondence with
the inflection points of $C$.
In particular,
the set $\Sing B$ has a natural structure of the $2$-dimensional affine space
over $\F_3$.
Each $Z$-splitting conic $\Gm$  passes through $6$ points of $\Sing B$,
and the complement $\Sing B\setminus (\Sing B\cap \Gm)$ is an affine line of $\Sing B$.
Thus there is a one-to-one correspondence 
between the set of $Z$-splitting conics for $B$,
and the set of affine lines of $\Sing B$.
\par
\medskip
Suppose that $z_2(\lattype (B))=6$.
Then $B$ is a union of three smooth conics with $R_B=6A_3$.
The configuration type of simple  sextics $B$ with $\degs B=[2,2,2]$ and $R_B=6A_3$ consists 
of a single lattice type,
and  the group $G_B$ is isomorphic to $\Z/4\Z\times \Z/4\Z$.
Moreover the class orders of the six $Z$-splitting conics for $B$ are all  $4$.
Let $B=C_1+C_2+C_3$ be  a simple sextic in this lattice type.
There exists a one-to-one correspondence between the six $Z$-splitting conics for $B$ and the six tacnodes of $B$,
which is described as follows.
Let $P\in \Sing B$ be a tacnode that is a tangent  point of two distinct conics $C_i$ and $C_j$.
Then there exists a unique $Z$-splitting conic that does not pass through $P$
but is tangent to both of $C_i$ and $C_j$ at the other tacnode $P\sprime\in \Sing B$ on $C_i\cap C_j$,
and passes through  the other four tacnodes on $C_k$ $(k\ne i, j)$. 
\subsection{Degeneration of   $Z$-splitting conics}
Consider the following two lattice types of simple sextics:
\begin{eqnarray*}
\lattype_{\smallAAAA, l}&=&\lattype_{lin, 6}\quad\quad\, (R_B=3A_5,\;\; \degs B=[3,3], \;\;  z_1(\lattype_{\smallAAAA, l})=1), \quand\\
\lattype_{\smallaaaa, c}&=&\lattype_{con, 3}\quad\quad (R_B=6A_2,\;\;  \degs B=[6], \;\;\;\;\;  z_2(\lattype_{\smallaaaa, c})\,=1).
\end{eqnarray*}
It is well-known that any member of $\lattype_{\smallaaaa, c}=\lattype_{con, 3}$
is defined by an equation of $(2, 3)$-torus type
$$
B\;:\;f^3+g^2=0\qquad (\deg f=2, \;\deg g=3),
$$
while it is easy to see that any member of $\lattype_{\smallAAAA, l}=\lattype_{lin, 6}$
is defined by an equation of $(2, 6)$-torus type
$$
B\sprime \;:\;l^6+g^2=0\qquad (\deg l=1, \;\deg g=3).
$$
When the quadratic polynomial $f$ degenerates into $l^2$,
then $B$ degenerates into $B\sprime$ and the $Z$-splitting conic
$\Gm=\{f=0\}$ for $B$ degenerates into the double of the $Z$-splitting  line $\Gm\sprime=\{l=0\}$ for $B\sprime$.
Therefore we can regard the $Z$-splitting  line $\Gm\sprime$ as the reduced part of a \emph{non-reduced} $Z$-splitting conic.
\par
\medskip
It seems that any $Z$-splitting line can be obtained as the reduced part of a non-reduced $Z$-splitting conic as above.
For example,
it is quite plausible  that there may exist  the following specializations from 
the lattice type $\lattype$ with $z_2(\lattype)=1$ to the lattice type $\lattype\sprime$ with $z_1(\lattype\sprime)=1$
that makes the $Z$-splitting conic for $\lattype$ to the  double of the $Z$-splitting line for $\lattype\sprime$:
$$
\begin{array}{c|c}
\lattype & \lattype\sprime \mystrutd{4pt}\\
\hline \mystruthd{18pt}{14pt}
\parbox{6.2cm}{$\lattype_{\smallbbbb, c}=\lattype_{con, 4}\\ (R_B=2A_1+4A_3,\;  \degs B=[2, 4])$}
&
\parbox{5.8cm}{$\lattype_{\smallBBBB, l}=\lattype_{lin, 8}\\ (R_B=A_3+2A_7,\;  \degs B=[2, 4])$} \\
\hline \mystruthd{18pt}{14pt}
\parbox{6.2cm}{$\lattype_{\smallcccc, c}=\lattype_{con, 5}\\ (R_B=4A_4,\;  \degs B=[6])$}
&
\parbox{5.8cm}{$\lattype_{\smallCCCC, l}=\lattype_{lin, 10}\\ (R_B=2A_4+A_9,\;  \degs B=[1,5])$} \\
\hline \mystruthd{18pt}{14pt}
\parbox{6.2cm}{$\lattype_{\smalldddd, c}=\lattype_{con, 6}\\ (R_B=2A_1+2A_2+2A_5,\; \degs B=[2,4])$}
&
\parbox{5.8cm}{$\lattype_{\smallDDDD, l}=\lattype_{lin, 12}\\ (R_B=A_3+A_5+A_{11},\;  \degs B=[2,4])$} \\

\end{array}
$$
The adjacency of $ADE$-types in these conjectural  specializations are all of the type $2A_l\to A_{2l+1}$.
However the existence of these specializations has not yet been confirmed.
\subsection{$Z$-splitting curves in positive characteristics}
The study of $Z$-splitting curves 
has stemmed  from the research of supersingular $K3$ surfaces in characteristic $2$.
In~\cite{MR2129248}, 
we have developed the theory of $Z$-splitting curves 
for purely inseparable double covers of $\Pt$ 
by supersingular $K3$ surfaces  in characteristic $2$.
The configuration of $Z$-splitting curves for such a covering
is described by a  binary linear code of length $21$.
Using this theory,
we have described the stratification of the moduli of polarized supersingular $K3$ surfaces
of degree $2$ in characteristic $2$ by the Artin invariant.

\par
\medskip
Using the structure theorem of the N\'eron-Severi lattices of supersingular $K3$ surfaces
by Rudakov-Sharfarevich~\cite{MR633161},
we can construct the theory of $Z$-splitting curves for supersingular double sextics in odd characteristics.
Note that every  supersingular $K3$ surface 
can be obtained as double sextics~\cite{MR2059747, MR2036331}.

\bibliographystyle{plain}

\def\cprime{$'$} \def\cprime{$'$} \def\cprime{$'$} \def\cprime{$'$}
  \def\cprime{$'$}

\end{document}